\documentclass[12pt, reqno, twoside]{amsart}
\usepackage{paper_style}

\usepackage[csfont, paper]{std_math}

\usepackage{local}

\IfFileExists{./\foldername/0_abstract.tex}{}{}

\hypersetup{pdfauthor={Segev Gonen Cohen}
	,pdftitle=Actions of lattices in p-adic and S-arithmetic groups on manifolds
	,urlcolor=blue
	,citecolor=red
    ,linkcolor=blue
	,colorlinks=true
}
\begin{document}
		
	\title{Actions of lattices in $p$-adic and $S$-arithmetic groups on manifolds}
\author[S.~Gonen Cohen]{Segev Gonen Cohen}
\address{Department of Mathematics, ETH Zürich, Zürich, Switzerland}
%\address{ETH Zurich, Department of Mathematics, Rämistrasse 101, 8092 Zurich, Switzerland}
\email{segev.gonencohen@math.ethz.ch}

\thanks{The author acknowledges support of the Institut Henri Poincaré (UAR 839 CNRS-Sorbonne Université), LabEx CARMIN (ANR-10-LABX-59-01), FNS (Fonds National Suisse, no. 10.003.145), and the ANR (Agence Nationale de la Recherche, ANR-24-CE93-0016)} 

%\keywords{}

\date{\today.}
	
\begin{abstract}
    We prove that an action by $C^1$-diffeomorphisms of a lattice in a simple $p$-adic group on a compact manifold is finite, provided that the dimension of the manifold is less than the rank of the group. We extend this statement to lattices in totally disconnected $S$-arithmetic groups, where the critical dimension is the maximal rank of its simple factors. This uses the machinery developed by Brown, Fisher, and Hurtado. 
\end{abstract} 
	\maketitle

    \tableofcontents

%%%%%%%%%%%%%%%%%%%%%%%%%%%%%%%%%%%%%%%%%%%%%%%%%%%%%%%%%%%
    \renewcommand{\thetheo}{\Alph{theo}}
\setcounter{theo}{0}

\section{Introduction}

In breakthrough work \cite{brown_zimmers_2020}, Aaron Brown, David Fisher, and Sebastian Hurtado proved that lattices in higher rank simple real Lie groups admit no infinite smooth actions on compact manifolds of low enough dimension. In this paper we extend their methods to obtain results for lattices in more general $S$-arithmetic groups, including to lattices in higher rank $p$-adic groups. Our main result is that below a certain dimension, all actions are finite. 

\myindent In all that follows $\mathcal{V}$ denotes the set of valuations of $\mathbb{Q}$, this is the set of prime numbers and the Archimedean place $\{\infty\}$. For any $p\in \mathcal{V}$, $\Qp$ denotes the corresponding completion of $\mathbb{Q}$, with the understanding that $\mathbb{Q}_\infty = \R$. For any manifold, $\Diff^1(M)$ denotes the group of $C^1$-diffeomorphisms. Throughout, all manifolds will be assumed connected and without boundary.

\begin{theo} 
\label{theo:main_theorem}
    Let $S \subset \mathcal{V}\backslash\{\infty\}$ be a finite set, and for each $p\in S$ let $\mathbb{G}_p$ be a connected simple algebraic group defined over $\Qp$. Consider $G = \prod_{p\in S} \mathbb{G}_p(\mathbb{Q}_p)$ and an irreducible lattice $\Gamma \leq G$. Let $M$ be a compact real manifold and let $\alpha:\Gamma \to \Diff^1(M)$ be a homomorphism. Denote by $r = \max_{p\in S}\{\operatorname{rank}_{\Qp}(\mathbb{G}_p)\}$ and assume that $r \geq 2$. Suppose that either
    \begin{itemize}
        \item $\dim M < r$, or
        \item $\dim M = r$ and $\alpha(\Gamma)$ preserves a $C^1$-volume form on $M$. 
    \end{itemize}
    Then $\alpha(\Gamma)$ is finite.
\end{theo}

In Section~\ref{sect:proof_sketch} we use Margulis' Normal Subgroup Theorem to show that Theorem~\ref{theo:main_theorem} is equivalent to the following special case concerning simple groups.

\renewcommand{\thetheo}{\Alph{theo}$'$}
\addtocounter{theo}{-1}

\begin{theo}
    \label{theo:simple_p_adic_case}
    Let $k$ be a non-Archimedean local field of characteristic 0 and $\mathbb{G}$ a connected simple algebraic group defined over $k$. Consider $G = \mathbb{G}(k)$ and a lattice $\Gamma\leq G$. Let $M$ be a compact real manifold and let $\alpha:\Gamma \to \Diff^1(M)$ be a homomorphism. 
    Assume that $r = \operatorname{rank}_k(\mathbb{G}) \geq 2$. Suppose that either
    \begin{itemize}
        \item $\dim M < r$, or
        \item $\dim M = r$ and $\alpha(\Gamma)$ preserves a $C^1$-volume form on $M$. 
    \end{itemize}
    Then $\alpha(\Gamma)$ is finite.
\end{theo}

\renewcommand{\thetheo}{\Alph{theo}}

In a similar vein, in the presence of a non-compact real factor we can state a related result, which we have not seen recorded in the literature. This follows by utilising a similar argument to restrict to the real place, and replacing our Theorem~\ref{theo:simple_p_adic_case} with \cite[Theorem 1]{brown_c1_2022}. 

\begin{theo}
\label{theo:with_real_place}
    Let $S \subset \mathcal{V}$ be a finite set containing the Archimedean place, and for each $p\in S$ let $\mathbb{G}_p$ be a connected simple algebraic group defined over $\Qp$. Consider $G = \prod_{p\in S} \mathbb{G}_p(\mathbb{Q}_p)$ and a uniform irreducible lattice $\Gamma \leq G$. Let $M$ be a compact real manifold and let $\alpha:\Gamma \to \Diff^1(M)$ be a homomorphism. Let $r = \operatorname{rank}_\R(\mathbb{G}_\infty)$, and assume that $r \geq 2$. Suppose that either
    \begin{itemize}
        \item $\dim M < r$, or
        \item $\dim M = r$ and $\alpha(\Gamma)$ preserves a $C^1$-volume form on $M$. 
    \end{itemize}
    Then $\alpha(\Gamma)$ is finite.
\end{theo}

We defer further discussion of the critical dimensions for later in the introduction. Our method of proof follows closely the aforementioned work of Brown-Fisher-Hurtado \cite{brown_zimmers_2020}, as well as their subsequent works \cite{BrownFisherHurtado2020, BrownFisherHurtado2021preprint}. Some key steps in the proof are heavily inspired by \cite{brown_c1_2022}, allowing for the $C^1$-regularity.

\subsection{\normalfont{\textit{Brief history}}}
In a series of papers in the 1980's, most notably \cite{Zimmer1984Arithmetic, Zimmer1986Actions}, Zimmer stated a series of conjectures regarding the classification of actions of higher rank semisimple Lie groups on compact manifolds, including the $S$-arithmetic setting. 

\myindent Significant progress has been made in the case of lattices in higher rank real simple groups (see for example the surveys by Fisher \cite{Fisher2019GroupsActingOnManifoldsZimmer, fisher_recent_2019} and Brown \cite{brown_entropy_2019}), including the full resolution of Zimmer's conjectures for some split simple real Lie groups \cite{brown_zimmers_2020, BrownFisherHurtado2020, BrownFisherHurtado2021preprint}. Considerably less attention has been given to groups over other local fields; although $p$-adic analysis has at times been used to great effect, see for example \cite{CantatXie2018}. We highlight some motivations for studying these other fields, and some notable prior results.

\myindent Classical motivation arises from the \emph{Hilbert-Smith conjecture}, which asserts that the only locally compact groups that can act continuously and faithfully on a manifold $M$ are Lie groups \cite{Smith1941}. It was shown in \cite{Newman1931, Smith1941Newman} that the conjecture is equivalent to showing that no continuous and faithful action exists for the group of $p$-adic integers $\Z_p$. This holds for actions by diffeomorphisms \cite{BochnerMontgomery1946}, and actions by H\"{o}lder homeomorphisms with sufficiently large exponent \cite{RepovsScepin1997,Maleshich1997}. The full Hilbert-Smith conjecture for actions by homeomorphisms is known for manifolds of dimensions 1 and 2 (see \cite[pg. 223, 249]{MontgomeryZippin1955}), and dimension 3 \cite{Pardon2013HilbertSmith}. A negative answer to the Hilbert-Smith conjecture could potentially be used to construct exotic actions of lattices in high rank $p$-adic groups.

\myindent For higher rank lattices (over any local field of characteristic 0), fundamental results in the more modern study of actions on manifolds include Margulis' superrigidity and Zimmer's cocycle superrigidity, which in many cases give a complete classification of finite dimensional representations and linear cocycles; see \cite{Zimmer1984, Margulis1991}. Both of these results play key roles in the sequel.

\myindent When the manifold being acted on is of dimension 1, so it is the circle or the real line, the fact that actions must be finite is known in many cases, see \cite{Witte1994ArithmeticGroups1Manifolds, Ghys1999Actions, BurgerMonod2002, Navas2002KazhdanCircle} and the surveys \cite{navas2011groups, morris_can_2009}. We mention that recently Deroin and Hurtado \cite{deroin_non_2020} proved in full generality that actions of lattices in higher rank simple real groups by homeomorphisms on the circle are finite; this was extended by Witte-Morris to lattices in higher rank simple $p$-adic groups \cite{morris_non-left-orderability_2024}. 

\myindent For dimension 2 already much less was known; and the only results that apply to $p$-adic lattices are those using only Kazhdan's property (T). For example Zimmer showed in \cite{zimmer_kazhdan_1984} that if a $C^1$-action of a property (T) group on a compact surface has a finite orbit, then in fact the action is finite. 

\myindent Similarly for higher dimensions we are unaware of results that hold specifically for $p$-adic or $S$-arithmetic groups rather than general property (T) groups. See for example \cite{zimmer_spectrum_1991, BekkaDeLaHarpeValette2008hb} for many results about rigidity of actions on manifolds for property (T) groups.

\myindent A key step in the proof of Theorem~\ref{theo:simple_p_adic_case} is a type of measure rigidity statement, Theorem~\ref{theo:G_invariance}. There is a rich history of results in this vein, especially in the homogeneous setting. We mention in particular the survey \cite{Pisa} and the references therein.

\subsection{\normalfont{\textit{Dimension numerology}}} 
We introduce some of the dimensions that arise in studying actions of lattices on manifolds, and state some natural conjectures. We will write $\Diff^\bullet(M)$ to indicate when we believe these results should hold for any regularity, but already the case of smooth (or even real-analytic) diffeomorphisms would be interesting. 

\myindent Let $S\subset \mathcal{V}$ be a finite set including $\infty$, and  $G = \prod_{p\in S}\mathbb{G}_p(\mathbb{Q}_p)$ an $S$-arithmetic group where each $\mathbb{G}_p$ is a connected simple algebraic group defined over $\mathbb{Q}_p$. A corollary of Margulis superrigidity and subsequent works is that the real representation theory of any irreducible lattice $\Gamma \leq G$ is controlled, up to bounded error, by that of $\mathbb{G}_\infty(\R)$ (see for example \cite{Zimmer1986Actions} and the references therein). Furthermore, letting $\mathfrak{g}$ be the Lie algebra of $\mathbb{G}_\infty(\R)$, the bounded representations are controlled by the representations of the compact real form of $\mathfrak{g}^\C$ (we recall this fact and its proof in Section~\ref{sect:compact_codomain}). 

\myindent It makes sense therefore to assume that the \emph{infinite} dimensional representations of $\Gamma$ are similarly constrained, indeed Kazhdan's property (T) can be viewed as a first statement in this direction. Given that homomorphisms $\alpha:\mathbb{G}_\infty(\R) \to \Diff^\bullet(M)$ are somewhat well understood (see for example \cite{stuck_low_1991}), it can be hoped that certain results can be carried over to the lattice $\Gamma$.  While this is not always true (see some examples in \cite{Fisher2019GroupsActingOnManifoldsZimmer}), this is the general philosophy of the Zimmer programme. 

\myindent There are two natural quantities to consider. For a real semisimple Lie group $H$ with Lie algebra $\mathfrak{h}$, we denote by $d(H)$ the minimal codimension of a proper subalgebra of the compact real form of $\mathfrak{h}_\C$. We remark that for a simple real Lie group $H$ of rank $\geq 2$, we have that $d(H) > \operatorname{rank}_\R(H)$ (see Proposition~\ref{prop:d_vs_rank}). Similarly, for a real semisimple Lie group $H$ define $\nu(H)$ to be the minimal codimension of a proper parabolic subalgebra of $\mathfrak{h}$. For a general $S$-arithmetic group $G = \prod_{p\in S}\mathbb{G}_p(\mathbb{Q}_p)$ where $\infty\in S$ we set $d(G):=d(\mathbb{G}_\infty(\R))$ and $v(G):= v(\mathbb{G}_\infty(\R))$, where $\mathbb{G}_\infty$ denotes the real factor of $G$ (see Section~\ref{sect:S-arithmetic} for the definition in the case that $\infty \notin S$).

In the spirit of the Zimmer programme, we record the following conjecture. We are unaware if it appears in print, although it seems to be the natural extension of the standard conjecture to the $S$-arithmetic context.

\begin{conjecture}
\label{conj:S_Arithmetic}
    Let $\Gamma \leq G$ be as above, and $M$ a compact real manifold. Suppose that $\operatorname{rank}_{\Qp}(\mathbb{G}_p) \neq 1$ for every $p$. Let $\alpha:\Gamma \to \Diff^\bullet(M)$ be a group homomorphism. Then
    \begin{enumerate}
        \item If $\dim(M) < \nu(G)$ then $\alpha(\Gamma)$ preserves a Riemannian metric;
        \item If $\dim(M) < \min\{\nu(G),d(G)\}$ then $\alpha(\Gamma)$ is finite.
    \end{enumerate}
\end{conjecture}

This is of course vacuous in the case that $\operatorname{rank}_{\Qp}(\mathbb{G}_p) =0$ for all $p\in S$. There are similar conjectures in the case of volume preserving actions to be made, see for example the discussion in \cite{Cantat2017Bourbaki1136}. See also \cite[Tables 1,2, and 3]{Cantat2017Bourbaki1136} for the values of $v(G)$ and $d(G)$ for all simple real Lie groups.

\myindent We isolate in particular the case that $\mathbb{G}_\infty(\R)$ is compact, expecting strong conclusions to hold in this case. 

\begin{que}
\label{que:compact_real}
    Let $\Gamma \leq G$ be as above, and suppose that $\operatorname{rank}_\R(\mathbb{G}_\infty)=0$. Is there a compact real manifold $M$ of dimension $\dim M < d(G)$ admitting an infinite action of $\Gamma$ by $C^1$-diffeomorphisms? Is there a compact real manifold $M$ of \textbf{any} dimension admitting a non-isometric action of $\Gamma$ by $C^1$-diffeomorphisms? 
\end{que}

We suspect that this question has a negative answer even in the case that all factors of $G$ have rank 1, although this seems currently out of reach. Recall that in Theorem~\ref{theo:main_theorem} we require at least one factor to be of higher rank to achieve any non-trivial dimension bound. We give two different ways to view actions of such $\Gamma$ on a manifold of dimension $d(G)$ in Section~\ref{sect:S-arithmetic}.

\myindent While resolving Question~\ref{que:compact_real} fully seems currently out of reach, one might hope that under an additional assumption on the $S$-rank we might be able to say more. Recall that the $S$-rank is defined to be $\operatorname{rank}_S(G) := \sum_{p\in S}\operatorname{rank}_{\mathbb{Q}_p}(\mathbb{G}_p)$. We record one such conjecture, which seems more attainable with current tools.  

\begin{conjecture}
    Let $\Gamma \leq G$ be as above, and $M$ a compact real manifold. Suppose that $\operatorname{rank}_{\Qp}(\mathbb{G}_p) \neq 1$ for every $p$ and that $\operatorname{rank}_\R(\mathbb{G}_\infty)=0$. Let $\alpha:\Gamma \to \Diff^\bullet(M)$ be a group homomorphism. If $\dim M < \operatorname{rank}_S(G)$, then $\alpha(\Gamma)$ preserves a Riemannian metric on $M$.
\end{conjecture}

A striking feature of Theorem~\ref{theo:main_theorem} is that it depends only on the \emph{maximal} rank, and in particular it holds even in the case where some factors have rank 1. We expect a similar result to be true for real semisimple groups also, the following conjecture was suggested to the author by Simon Machado.

\begin{conjecture}
\label{conj:max_rank}
    Let $G$ be a real semisimple group, and $\Gamma \leq G$ a uniform irreducible lattice. Let $r = \max_{i}\{\operatorname{rank}_{\R}(G_i)\}$ where the $G_i$ are the simple factors of $G$, and assume that $r \geq 2$. Suppose that either
    \begin{itemize}
        \item $\dim M < r$, or
        \item $\dim M = r$ and $\alpha(\Gamma)$ preserves a $C^1$-volume form on $M$. 
    \end{itemize}
    Then $\alpha(\Gamma)$ is finite.
\end{conjecture}

\subsection{\normalfont{\textit{Groups of $\mathbb{Q}$-points}}} 

A common motivation to consider $S$-arithmetic groups is that they allow us to say something concrete about the $\mathbb{Q}$-points of algebraic groups. In the setting of actions on manifolds we are unaware of any results in this vein other than \cite[Theorem 28]{BurgerMonod2002}, from which it can be shown that if the group of $\mathbb{Q}$-points of a higher rank simple $\mathbb{Q}$-group acts on a circle, the action must be trivial. 

\myindent We record the following consequence of Theorem~\ref{theo:simple_p_adic_case}, which is deduced along similar lines to Theorem~\ref{theo:main_theorem}. Recall that an (almost) simple group $\mathbb{G}$ defined over a field $k$ is said to be \emph{absolutely (almost) simple} if it remains simple under the algebraic closure $\overline{k}$.

\begin{theo}
\label{theo:Q_groups}
    Let $S \subset \mathcal{V}$ be an infinite set of places of $\mathbb{Q}$ including the Archimedean one. Let $\mathbb{G}$ be a connected algebraic group defined over $\mathbb{Q}$ that is absolutely simple. Let $\Gamma = \mathbb{G}(\Z_S)$, where $\Z_S := \Z[\frac{1}{p}:p\in S\backslash \{\infty\}]$.
    Let $M$ be a compact real manifold and let $\alpha:\Gamma \to \Diff^1(M)$ be a homomorphism. 
    
    If $\operatorname{rank}_\R(\mathbb{G}) = 0$, let $r= \operatorname{rank}_\C(\mathbb{G})$, otherwise let $r = \operatorname{rank}_\R(\mathbb{G})$. Assume that $r \geq 2$ and that $\operatorname{rank}_{\mathbb{Q}}(\mathbb{G}) = 0$ in both cases. Suppose that either
    \begin{itemize}
        \item $\dim M < r$, or
        \item $\dim M = r$ and $\alpha(\Gamma)$ preserves a $C^1$-volume form on $M$. 
    \end{itemize}
    Then $\alpha(\Gamma)$ is finite.
\end{theo}

In a similar vein, we make the following conjecture.

\begin{conjecture}
    Let $\mathbb{G}$ be a connected algebraic group defined over $\mathbb{Q}$ that is absolutely simple, and let $\Gamma = \mathbb{G}(\mathbb{Q})$. Let $M$ be a compact real manifold and let $\alpha:\Gamma \to \Diff^1(M)$ be a homomorphism. Assume that $\operatorname{rank}_\R(\mathbb{G})=0$ and that $\operatorname{rank}_\C(\mathbb{G})\geq 2$.
    
    Then $\alpha(\Gamma)$ preserves a measure on $M$, and in fact preserves a Riemannian metric. If in addition the $\alpha(\Gamma)$-orbit of some point is dense in $M$, then $M = \mathbb{G}(\R)/K$ for some closed subgroup $K \leq \mathbb{G}(\R)$. 
\end{conjecture}

\subsection*{Acknowledgements} I would like to especially thank Homin Lee for introducing me to this problem and for many helpful discussions, Simon Machado for his significant help and suggestions, and Manfred Einsiedler for his patient explanations and intuition. Additionally I would like to thank Menny Aka, Aaron Brown, Marc Burger, David Fisher, Michael Glasner, and Andreas Wieser for useful discussions, comments, and explanations that were vital for this work.

    \setcounter{theo}{0}
\renewcommand{\thetheo}{\arabic{section}.\arabic{theo}}

    \section{Proof outline}
\label{sect:proof_sketch}

We explain first the reduction to Theorem~\ref{theo:simple_p_adic_case}. Recall that Margulis' Normal Subgroup Theorem \cite[Theorem VIII.2.6]{Margulis1991} says that whenever $\Gamma \leq G$ is an irreducible lattice such that $\operatorname{rank}_S(G) \geq 2$, and $N \lhd \Gamma$ is a normal subgroup, then either $N$ or $\Gamma/N$ is finite.

\begin{proof}[Proof of Theorems~\ref{theo:main_theorem} and \ref{theo:with_real_place} assuming Theorem~\ref{theo:simple_p_adic_case}]
We first prove Theorem~\ref{theo:main_theorem}, so let $S \subset \mathcal{V}\backslash \{\infty\}$ be a finite set of places of $\mathbb{Q}$. Let $p \in S$ be a place such that $\operatorname{rank}_{\mathbb{Q}_p}(\mathbb{G}_p)=r$ is maximal amongst the $p\in S$. For every $q \neq p$ let $K_{q} \leq \mathbb{G}_{q}(\mathbb{Q}_{q})$ be a compact-open subgroup.
Consider then the group
\begin{equation*}
    \Gamma' := \Gamma \cap \left(\mathbb{G}_p(\mathbb{Q}_p)\times \prod_{q \neq p} K_{q}\right).
\end{equation*}
If $\pi_p:G\to \mathbb{G}_p(\mathbb{Q}_p)$ denotes the projection map, $\pi_p(\Gamma') \leq \mathbb{G}_p(\mathbb{Q}_p)$ is a lattice. Let $1_p$ denote the identity element in $\mathbb{G}_p(\mathbb{Q}_p)$, and notice that since $\Gamma$ is discrete the group
\begin{equation*}
\Gamma \cap \left(\{1_p\}\times \prod_{q \neq p} K_{q}\right)
\end{equation*}
is finite. So we can assume by taking a finite index subgroup $\Gamma'' \leq \Gamma'$ that $\pi_p|_{\Gamma''}$ is injective, and so $\Gamma''$ itself can be identified with a lattice in $\mathbb{G}_p(\mathbb{Q}_p)$. Consider now the action of $\Gamma''$ on $M$, namely the restriction $\alpha|_{\Gamma''}:\Gamma''\to \Diff^1(M)$. By assumption on $\dim M$, Theorem~\ref{theo:simple_p_adic_case} says that $\alpha(\Gamma'')$ is finite, and so in particular $\alpha$ has infinite kernel. By Margulis' Normal Subgroup Theorem the result follows.

\myindent In Theorem~\ref{theo:with_real_place} the real factor $\mathbb{G}_\infty(\R)$ is not assumed to be compact and satisfies $\operatorname{rank}_\R(\mathbb{G}_\infty)\geq 2$. We set $p=\infty$ and repeat the above argument to obtain a uniform lattice in $\mathbb{G}_\infty(\R)$, which satisfies the assumptions of \cite[Theorem 1]{brown_c1_2022}. So again $\alpha$ has infinite kernel and the result follows similarly.
\end{proof}

\begin{rem}
    The statements of Theorems~\ref{theo:main_theorem} and \ref{theo:with_real_place} can be strengthened in a few directions. We choose not to state them in full generality for simplicity, instead being content with brief indications as follows:
    \begin{enumerate}[label = {(\roman*)}]
        \item One can replace each $\mathbb{G}_p$ with a semisimple group defined over $\mathbb{Q}_p$. In the case where $\infty \notin S$ it is clear from the reduction to Theorem~\ref{theo:simple_p_adic_case} above that one can let the critical dimension be $r = \max_{p\in S, i}\{\operatorname{rank}_{\mathbb{Q}_p}(\mathbb{G}_{p,i})\}$ where the $\mathbb{G}_{p,i}$ are the simple factors of $\mathbb{G}_p$.
        \item If $\infty\in S$, then as above one reduces to the case of a uniform lattice in $\mathbb{G}_\infty(\R)$ and considers $r = \min_{i}\{\operatorname{rank}_\R(\mathbb{G}_{\infty,i})\}$ (provided this is $\geq 2$). Conjecture~\ref{conj:max_rank} would allow us to strengthen the $\min$ to a $\max$, and thus allow for the presence of rank 1 factors.
        \item Both theorems can be stated over general global fields of characteristic 0, but this reduces to the above points by restriction of scalars.
        \item The dimension bound in Theorem~\ref{theo:with_real_place} can be increased and $\Gamma$ can be taken to be non-uniform (at the price of working with regularity $C^2$) by replacing \cite[Theorem 1]{brown_c1_2022} with the main results of \cite{brown_zimmers_2020, BrownFisherHurtado2020, BrownFisherHurtado2021preprint, an_brown_zhang_2024_zimmers_conjecture_nonsplit}. 
    \end{enumerate}

\end{rem}

The same comments also apply to the statement of Theorem~\ref{theo:Q_groups}. The argument for it is similar, we record it here for the readers' benefit. 

\begin{proof}[Proof of Theorem~\ref{theo:Q_groups} assuming Theorem~\ref{theo:simple_p_adic_case}]
    By Margulis' Normal Subgroup Theorem for infinitely many places (\cite[Theorem VIII.2.6]{Margulis1991}) it again suffices to show that the kernel of $\alpha$ is infinite. In the case that $\operatorname{rank}_\R(\mathbb{G})\geq 2$ we consider $\Gamma'=\mathbb{G}(\Z)$ and apply  \cite[Theorem 1]{brown_c1_2022} to $\alpha|_{\Gamma'}$. Otherwise, we note that for all but finitely many places we have that $\operatorname{rank}_{\Qp}(\mathbb{G}) = r$. Choose such a place $p$, consider $\Gamma' = \mathbb{G}(\Z[\tfrac{1}{p}])$ and apply Theorem~\ref{theo:simple_p_adic_case} to $\alpha|_{\Gamma'}$.
\end{proof}

We now outline the general strategy of the proof of Theorem~\ref{theo:simple_p_adic_case}. This is part of the scheme developed in \cite{brown_zimmers_2020} and the subsequent works \cite{BrownFisherHurtado2020, BrownFisherHurtado2021preprint, brown_c1_2022}, and has three main steps.

\subsection*{Step 1: Uniform subexponential growth of derivatives}

Since $\Gamma$ is uniform it is finitely generated, so we can fix a finite generating set of $\Gamma$ and for $\gamma \in \Gamma$ denote by $\ell(\gamma)$ the corresponding word length. In all that follows we assume that $M$ is equipped with a smooth Riemannian metric in the background, the particular choice will not matter. We recall the following key definition.

\begin{defn}
    We say that $\alpha$ has \emph{uniform subexponential growth of derivatives} if for every $\varepsilon > 0$ there is a constant $C_\varepsilon > 0$ such that for any $\gamma \in \Gamma$ we have
    \begin{equation*}
        \sup_{y\in M}\|D_y\alpha(\gamma)\| \leq C_\varepsilon e^{\varepsilon \ell(\gamma)}
    \end{equation*}
\end{defn}

Clearly this property is independent of the choice of the generating set of $\Gamma$ and the underlying Riemannian metric on $M$. The main result of step 1 is the following:

\begin{theo}
\label{theo:subexp_growth}
Let $M, G, \Gamma$, $r$, and $\alpha$ be as in the statement of Theorem~\ref{theo:simple_p_adic_case}. Suppose that either
    \begin{itemize}
        \item $\dim M < r$, or
        \item $\dim M = r$ and $\alpha(\Gamma)$ preserves a $C^1$-volume form on $M$.
    \end{itemize}
    Then $\alpha$ has uniform subexponential growth of derivatives.
\end{theo}

The proof of Theorem~\ref{theo:subexp_growth} will occupy the majority of this paper, as we now outline. We first construct the \emph{suspension space} $M^\alpha$, which is a bundle over $G/\Gamma$ with fibers $M$ and a natural $G$-action and $G$-equivariant projection $\pi:M^\alpha\to G/\Gamma$. The usual derivative cocycle
$$D:\Gamma \times M \to \GL(\dim M, \R): (\gamma,y) \mapsto D_y\alpha(\gamma)$$
extends to a cocycle
$$\mathcal{D}:G\times M^\alpha \to \GL(\dim M, \R).$$

The definitions of $D$ and hence $\mathcal{D}$ depend on a preferred choice of charts; this will play no role in the sequel. The extension to $\mathcal{D}$ depends in addition on a choice of fundamental domain for $\Gamma$ in $G$, see Section~\ref{sect:Suspension_construction} for the explicit construction.

\myindent Many of the key properties of the $\Gamma$-action on $M$ can now be related to properties of the $G$-action on $M^\alpha$. Most notably, an action of $\Gamma$ without uniform subexponential growth of derivatives will be related to the existence of an $A$-invariant measure $\mu$ on $M^\alpha$ with a nonzero top Lyapunov exponent (see Section~\ref{sect:smooth_dynamics} for the relevant definitions). Here $A$ denotes a maximal $k$-split torus of $G$. All measures will be assumed to be Borel.

\begin{prop}
\label{prop:exists_positive_lyapunov}
    Suppose $\Gamma$ fails to have uniform subexponential growth of derivatives. Then there is some $a\in A$ and an $A$-invariant, $A$-ergodic measure $\mu$ on $M^\alpha$ such that the top Lyapunov exponent is $\lambda_+(a,\mu,\mathcal{D})>0$.
\end{prop}

We then average this measure $\mu$ along suitably chosen unipotent subgroups using Ratner's Theorems \cite{ratner_raghunathans_1995, margulis_invariant_1994} on its projection to $G/\Gamma$. By doing this we obtain a measure that also has a nonzero Lyapunov exponent, with the added property that its projection under $\pi$ is the Haar measure on $G/\Gamma$, which we denote by $m_{G/\Gamma}$.

\begin{prop}
\label{prop:average_to_haar}
    Suppose $\Gamma$ fails to have uniform subexponential growth of derivatives. Then there is some $a\in A$ and an $A$-invariant, $A$-ergodic measure $\mu$ on $M^\alpha$ such that $\lambda_+(a,\mu,\mathcal{D})>0$, and such that $\pi_*\mu = m_{G/\Gamma}$.
\end{prop}

This proposition resembles \cite[Proposition 4.7]{brown_zimmers_2020}, and is one of the key innovations in that work. Once the rank is $\geq 3$, one needs to average over unipotent subgroups of dimension $>1$, using a generalisation of Ratner's equidistribution Theorem due to Shah \cite[Corollary 1.3]{Shah1994}. We are unaware of the existence of a version of this in the literature for $p$-adic or $S$-arithmetic groups (even though there is little doubt it holds). To get around this we use some combinatorics of the roots, and specifically the derived series of Borel subgroups, see Lemma~\ref{lem:solvalble_averaging}. 

We now use the Avila-Viana invariance principle from \cite{avila_extremal_2010}, as in \cite{brown_c1_2022}, to obtain extra invariance for the measure $\mu$. This is where the key dimension restriction enters. The following is a special case of Theorem~\ref{theo:invariance_under_stanble_and_unstable_horosphericals} which is sufficient for our purposes.

\begin{theo}
    \label{theo:G_invariance}
    Let $M, G, \Gamma$, $r$, and $\alpha$ be as in the statement of Theorem~\ref{theo:simple_p_adic_case}. Let $\mu$ be an $A$-invariant and $A$-ergodic probability measure on $M^\alpha$ such that $\pi_*\mu = m_{G/\Gamma}$. Suppose that either
    \begin{itemize}
        \item $\dim M < r$, or
        \item $\dim M = r$ and $\alpha(\Gamma)$ preserves a $C^1$-volume form on $M$.
    \end{itemize}
    Then $\mu$ is $G$-invariant.
\end{theo}

Theorem~\ref{theo:subexp_growth} follows by a standard argument from Proposition~\ref{prop:average_to_haar}, Theorem~\ref{theo:G_invariance}, and Zimmer's Cocycle Superrigidity Theorem; we recall this for the readers' benefit.

\begin{proof}[Proof of Theorem~\ref{theo:subexp_growth}]
    We prove this by contradiction. Suppose that $\alpha$ fails to have uniform subexponential growth of derivatives. Then by Proposition~\ref{prop:average_to_haar} there exists an $a\in A$ and an $A$-invariant $A$-ergodic probability measure $\mu$ on $M^\alpha$ such that $\lambda_+(a,\mu,\mathcal{D})>0$ and $\pi_*\mu = m_{G/\Gamma}$ is the Haar measure on $G/\Gamma$. By Theorem~\ref{theo:G_invariance} this measure is $G$-invariant.

    \myindent Let $m=\dim M$. We claim there exists a measurable map $\Phi:M^\alpha \to \GL(m,\R)$ and a compact $K\subset  \GL(m,\R)$ such that
    \begin{equation}
    \label{eq:compact_cocycle}
        \Phi(g\cdot z) (\mathcal{D}_zg) \Phi(z)^{-1} \in K
    \end{equation}
    for all $g\in G$ and almost every $z\in M^\alpha$. Indeed, by the version of Zimmer's cocycle superrigidity appearing in \cite[Theorem 1.4]{FisherMargulis2003} the cocycle $\mathcal{D}:G\times M^\alpha \to \GL(m,\R)$ is measurably cohomologous, up to compact error, to a continuous group homomorphism $G \to \GL(m,\R)$. Since $m  \leq \operatorname{rank}_{\C}(\mathbb{G}) < v(G)$, this representation must be trivial, proving the claim.

    \myindent Choose sets $X_k\subset M^\alpha$ with $\mu(X_k) \nearrow 1$ as $k\to \infty$, for which the norms of $\Phi$ and $\Phi^{-1}$ are bounded. It is clear from Poincaré recurrence for the $X_k$ and \eqref{eq:compact_cocycle} that for every $\varepsilon>0$ the set of $z\in M^\alpha$ such that  
    \begin{equation*}
        \liminf_{n\to \infty} \frac{1}{n}\log\|\mathcal{D}_zg^n\| \geq \varepsilon
    \end{equation*}
    has measure zero. This contradicts the existence of a nonzero top Lyapunov exponent, completing the proof.
\end{proof}

\subsection*{Step 2: Strong Property (T)}
\label{sect:strong_T}

We recall \cite[Corollary B]{brown_c1_2022}, which is an adaptation to the $C^1$-setting of \cite[Theorem 2.9]{brown_zimmers_2020}. It uses strong Property (T), which was introduced by Lafforgue \cite{Laf08, Laf09}. Lafforgue showed that $\SL_3(k)$ has strong property (T) for $k$ any non-Archimedean local field, and Liao \cite{Liao} extended this to $\Sp_4(k)$ and thus to all almost simple $k$-groups with $\operatorname{rank}_k\geq 2$. strong Property (T) is now known for all lattices in products of high-rank semisimple groups over local fields, see \cite{dLdlS15, dlS19}.

\myindent The following is a combination of \cite[Proposition 5]{brown_c1_2022} (which relies on strong property (T)) and \cite[Lemma 1]{brown_c1_2022} (which relies on the solution of the Hilbert-Smith conjecture for sufficiently Hölder actions as in \cite{RepovsScepin1997, Maleshich1997}), although it is not stated in this form.

\begin{theo}
\label{theo:image_is_compact}
    Let $\Gamma$ be a group with strong property (T), $M$ a compact manifold, and $\alpha:\Gamma \to \Diff^1(M)$ an action with uniform subexponential growth of derivatives.

    Then there is a compact Lie group $K \leq \operatorname{Homeo}(M)$ such that $\alpha(\Gamma) \leq K$.
\end{theo}

\subsection*{Step 3: Margulis Superrigidity with Compact Image} \label{sect:compact_codomain}

We first describe some consequences of Margulis Superrigidity when the image of a real representation is bounded. This is well-known but instructive to review. Let $\phi:\Gamma \to K$ be such a representation with infinite image, where $K$ is a compact real Lie group. By passing to a subgroup we may assume that $K = \overline{\phi(\Gamma)}$, and up to finite index $K$ is connected. By the classification of compact Lie groups $K$ is isogenous to $\mathbb{T}^d \times K_1 \times \cdots \times K_\ell$ where the $K_i$ are compact simple Lie groups. 

\myindent Using the fact that the abelianisation of $\Gamma$ is finite (which follows from property (T)) we have that $d = 0$, and so it clearly suffices to analyse the case $K = K_i$ is simple. By \cite[Lemma 6.1.6 and Lemma 6.1.7]{Zimmer1984} (see also \cite[page 10]{brown_zimmers_2020}) we can find some number field $\mathbb{F}$ and an $\mathbb{F}$-group $\mathbb{K}$ such that $K = \mathbb{K}(\R)$ and $\phi(\Gamma) \leq K(\mathbb{F})$.

\myindent By applying superrigidity to $\phi(\Gamma) \hookrightarrow \mathbb{K}(\mathbb{F}_\nu)$ every place $\nu$ of $\mathbb{F}$ that does not lie over $k$, it follows that up to finite index every element of $\phi(\Gamma)$ has entries in $\mathcal{O}_\mathbb{F}[1/\varpi]$, where $\varpi$ is a uniformiser of $k$.

\myindent By compactness of $K$ it also follows that $\phi(\gamma)$ is semisimple for every $\gamma \in \Gamma$, so in particular we can pick some $h \in \phi(\Gamma)$ of infinite order. Hence some eigenvalue $t$ of $h$ is not a root of unity, so by \cite[Lemma 4.1]{tits_free_1972} there is some place $\nu$ of $\mathbb{F}$ and an embedding $\sigma:\mathbb{F}\to \mathbb{F}_\nu$ with $|\sigma(t)|_\nu \neq 1$.

\myindent So the representation $\phi^\sigma$ obtained by applying $\sigma$ to each entry is not bounded, and $\phi^\sigma(\Gamma)$ is Zariski-dense in some non-compact simple group $\mathbb{K}(\mathbb{F}_\nu)$ defined over $\mathbb{F}_\nu$. By superrigidity therefore $\phi^\sigma$ virtually extends to $G$, and so $\mathbb{K}$ must be of the same absolute type as $G$, and hence of the same absolute type as $\mathbb{G}_\infty$.

\myindent In conclusion, if $\phi:\Gamma \to K$ is a representation with infinite image and such that $\overline{\phi(\Gamma)} = K$ is compact, then it must come from some isogeny $\mathbb{G}_\infty(\R) \times \cdots \times \mathbb{G}_\infty(\R) \to K$.

\myindent We also record the following fact involving the quantity $d$ defined in the introduction. Recall that a real semisimple Lie group $H$ is called \emph{isotypic} if all the simple factors of $H$ have the same absolute type. As discussed in Section~\ref{sect:S-arithmetic} it is a consequence of Margulis' Arithmeticity Theorem that if $G$ is a higher rank $S$-arithmetic group admitting an irreducible lattice, then $G$ is isotypic.

\begin{prop}
\label{prop:d_vs_rank}
    Let $H$ be an isotypic real semisimple Lie group. Then
    $$d(H) \geq 2\operatorname{rank}_\R(H_i)-1$$
    for all simple higher-rank factors $H_i$ of $H$. In particular,
    $$d(H) >  \operatorname{rank}_\R(H_i)$$
    for every simple factor $H_i$ of $H$.
\end{prop}

\begin{proof}
    We prove the first inequality, the second clearly holds for simple factors of rank 0 or 1. Let $H$ be a simple real Lie group, and $\mathfrak{h}_c$ the Lie algebra of a compact real form of the complexification $\mathfrak{h}_\mathbb{C}$. It is clear that $\operatorname{rank}_\R(H) \leq \operatorname{rank}_\C(H) = \operatorname{rank}_\R(\mathfrak{h}_c)$. Let $\mathfrak{k} \leq \mathfrak{h}_c$ be any proper subalgebra. By \cite[Corollary IV.5.4]{Bredon1972CompactTransformationGroups} we have that $\dim(\mathfrak{h}_c/\mathfrak{k}) \geq \operatorname{rank}_\mathbb{R}(\mathfrak{h}_c) + \operatorname{rank}_\mathbb{R}(\mathfrak{k})$. If $\mathfrak{k}$ is a maximal proper subalgebra of $\mathfrak{h}_c$ it follows that $\operatorname{rank}_\mathbb{R}(\mathfrak{k}) = \operatorname{rank}_\mathbb{R}(\mathfrak{h}_c)-1$ and so in $d(H) \geq 2\operatorname{rank}_\mathbb{R}(H) - 1$. The case of isotypic semisimple $H$ is easy by noting that $d(H) \geq d(H_i)$ for each simple factor $H_i$ of $H$ and using the above for the simple $H_i$.
\end{proof}

\begin{rem}
    The inequality given in Proposition~\ref{prop:d_vs_rank} is sharp in general, as witnessed by split groups of type $D_r$. See \cite[Table 2]{Cantat2017Bourbaki1136} for the values of $d$ for the classical non-compact simple Lie groups.
\end{rem}

\begin{proof}[Proof of Theorem~\ref{theo:simple_p_adic_case}]
    Suppose for contradiction that $\alpha(\Gamma)$ is infinite. By Theorem~\ref{theo:subexp_growth} we know that $\alpha$ has uniform subexponential growth of derivatives, and so Theorem~\ref{theo:image_is_compact} applies. Let $K$ be as in the conclusion of  Theorem~\ref{theo:image_is_compact}, so $\alpha(\Gamma)\subset K$, and so by the previous discussion (restricting to the case $\overline{\alpha(\Gamma)} = K$) any simple factor of $K$ must be isogenous to $\mathbb{G}_\infty(\R)$.

    \myindent Since by assumption $K$ is non-trivial and compact, there is a closed $K$-orbit in $M$ of the form $K\cdot x \cong K/C$ for some proper closed subgroup $C \leq K$, and thus $\dim(M) \geq \dim(K/C)\geq d(\mathbb{G}_\infty(\R)) > r$, which contradicts our assumptions on $\dim(M)$.
\end{proof}

\subsection*{Paper outline}

The paper is organised as follows. In Section~\ref{sect:S-arithmetic} we recall some fundamental facts about $S$-arithmetic groups and their lattices, and in Section~\ref{sect:smooth_dynamics} we recall important facts from smooth dynamics. In Section~\ref{sect:Averaging} we prove Proposition~\ref{prop:exists_positive_lyapunov} and Proposition~\ref{prop:average_to_haar}. In Section~\ref{sect:invariance} we prove Theorem~\ref{theo:G_invariance}, thus completing the proof of Theorem~\ref{theo:subexp_growth} and as outlined above of Theorem~\ref{theo:simple_p_adic_case}.

    \section{\texorpdfstring{$S$}{}-arithmetic groups}
\label{sect:S-arithmetic}

We introduce the class of groups that we are considering and recall some properties. Recall that $\mathcal{V}$ denotes the set of valuations of $\mathbb{Q}$, for any other number field $\mathbb{K}$ we denote its set of valuations by $\mathcal{V}_{\mathbb{K}}$, and its Archimedean ones by $\mathcal{V}_{\mathbb{K},\infty}$.

Let $S \subset \mathcal{V}$ be a finite set, and for each $p\in S$ let $\mathbb{G}_p$ be a connected semisimple $\mathbb{Q}_p$-group with finite center and no compact factors. Set $r_p = \operatorname{rank}_{\mathbb{Q}_p}(\mathbb{G}_p)$ (the dimension of the maximal subgroup of $\mathbb{G}_p$ that is diagonalisable over $\mathbb{Q}_p$). Let $G = \prod_{p\in S} \mathbb{G}_p(\mathbb{Q}_p)$ and set $s = \sum_{p\in S}r_p$, the \emph{$S$-rank} of $G$. We will call such a $G$, and the associated data, an \emph{$S$-arithmetic group}. A lattice $\Gamma \leq G$ will always be assumed to be irreducible and uniform. Recall that a lattice is irreducible if it has a dense projection to every proper factor of $G$. 

\myindent Margulis' Arithmeticity Theorem \cite[Theorem IX.1.11]{Margulis1991} tells us that as soon as $s \geq 2$, $\Gamma$ is arithmetic. That is, there is a number field $\mathbb{K}$ and some absolutely simple $\mathbb{K}$-group $\mathbb{H} \leq \SL_d$, a finite set of places $S' \subset \mathcal{V}_{\mathbb{K}}$ including the Archimedean ones $\mathcal{V}_{\mathbb{K},\infty}$, and a continuous homomorphism $\varphi:\prod_{\nu\in S'}\mathbb{H}(\mathbb{K}_\nu) \to G$ with compact kernel and finite index image such that $\varphi(\mathbb{H}(\mathcal{O}_{\mathbb{K},S'}))$ is commensurable to $\Gamma$. Here $$\mathcal{O}_{\mathbb{K},S'}:=\{t\in \mathbb{K}:\nu(t)\geq 0 \text{ for every } \nu\in S'\backslash \mathcal{V}_{\mathbb{K},\infty}\}$$
is the ring of $S'$-integers in $\mathbb{K}$. We denote by $G_\infty := \prod_{\nu\in \mathcal{V}_{\mathbb{K},\infty}} \mathbb{H}_\nu(\mathbb{K}_\nu)$, this is in particular a real Lie group. If $S \cap \mathcal{V}_{\mathbb{K},\infty} = \emptyset$, that is we started with a totally disconnected group, then $G_\infty$ is compact. An important consequence of the above, that was already used in Section~\ref{sect:proof_sketch}, is that the simple factors of the $\mathbb{G}_p$ have to be isotypic (that is, they have the same absolute type).

\myindent As explained in the proof sketch, in the sequel we need to prove only Theorem~\ref{theo:simple_p_adic_case}. So from now on (unless mentioned otherwise) $\mathbb{G} \leq \SL_d$ is a simple group defined over a non-Archimedean local field $k$ of characteristic $0$. In particular, $k$ is a finite extension of $\Qp$ for some $p$, and comes equipped with an absolute value $|\cdot|_k$ extending the usual $p$-adic one on $\Qp$. Let $\mathcal{O}_k = \{x\in k:|x|_k \leq 1\}$ be the ring of integers of $k$. We set $G = \mathbb{G}(k)$, and consider a lattice $\Gamma \leq G$. In fact the lattice must be uniform, see \cite{ratner1998} for a nice proof of this fact. Recall that we have the assumption that $r = \operatorname{rank}_k(\mathbb{G}) \geq 2$.

\myindent We indicate briefly two ways of viewing infinite actions of lattices in $p$-adic groups on manifolds in dimension $d(G)$; as noted in Question~\ref{que:compact_real}, we expect this to be the smallest dimension where such actions exist. For simplicity we explain these only for $\SL_n(\mathbb{Q}_p)$, other cases are similar. Here $d(\SL_n(\mathbb{Q}_p)) = d(\SU_n(\R)) = 2n-2$.

\begin{exmp} \label{exmp:lattice_in_SL(n,Qp)}
    Choose some $m \in \N$ coprime to $p$ such that $-m$ is a square in $\Qp$, and let $\sigma$ be the nontrivial involution of $\mathbb{Q}[\sqrt{-m}]$. Consider the unitary group $\mathbb{G}=\operatorname{SU}(h)$ attached to the hermitian form $g^tg^\sigma = \Id$, where $g$ is an $n\times n$ matrix and $g^\sigma$ means apply $\sigma$ component-wise. This is a connected simple $\mathbb{Q}$-group, with $\operatorname{rank}_\R(\mathbb{G})=0$. Therefore for $S' = \{\infty,p\}$ we see that
    \begin{equation*}
       \Gamma = \left\{g\in \SL_n(\Z[\tfrac{\sqrt{-m}}{p}])\mid g^Tg^\sigma = \Id\right\} = \mathbb{G}(\Z_{S'}) \leq \mathbb{G}(\R)\times \mathbb{G}(\Qp) \cong \SU_n(\R) \times \operatorname{SL}_n(\Qp).
    \end{equation*}
   Since $\SU_n(\R)$ is compact then $\Gamma$ is a uniform lattice in $\operatorname{SL}_n(\Qp)$. We have an action of $\mathbb{G}(\R)$, and hence an infinite one of $\Gamma$, on 
    \begin{equation*}
        \C\mathbb{P}^{n-1} \cong \operatorname{SU(n)}/\operatorname{S}(\operatorname{U}(n-1)\times \operatorname{U}(1)),
    \end{equation*}
    a real compact manifold of dimension $2n-2$.
\end{exmp}

\begin{exmp}
    As pointed out to the author by Andreas Wieser, there is another way to construct such actions. Let $\mathbb{C}_p$ be the algebraic closure of $\mathbb{Q}_p$, then there is an algebraic isomorphism $\mathbb{C}_p \cong \mathbb{C}$ which combined with inclusion gives a field embedding $\imath_p:\mathbb{Q}_p \hookrightarrow \C$. Applied component-wise we thus get an embedding $\phi_p:\SL_n(\mathbb{Q}_p) \hookrightarrow \SL_n(\C)$ and thus every nontrivial action of $\SL_n(\C)$ on a manifold gives a (discontinuous) one of $\SL_n(\mathbb{Q}_p)$ and a nontrivial one of any of its lattices.

    \myindent However, the identification $\mathbb{C}_p \cong \mathbb{C}$ has to map the algebraic closure $\overline{\mathbb{Q}}$ to itself, and so corresponds to some element of $\operatorname{Gal}(\overline{\mathbb{Q}}/\mathbb{Q})$. By Mostow Rigidity, as soon as $n\geq 3$ any lattice $\Gamma \leq \SL_n(\mathbb{Q}_p)$ has entries in $\SL_n(\overline{\mathbb{Q}})$ (up to finite index) and so by an argument similar to the one given in Section~\ref{sect:compact_codomain}, $\phi_p|_\Gamma$ actually extends to a homomorphism $\operatorname{SU}_n(\R) \to \SL_n(\C)$. In particular, the actions of $\Gamma$ obtained from $\phi_p$ are the standard ones we expect, despite them coming from somewhat exotic actions of $\SL_n(\mathbb{Q}_p)$.

    \myindent It would be interesting to attempt to modify this construction to obtain genuinely exotic actions of lattices in $\SL_n(\mathbb{Q}_p)$, possibly answering Question~\ref{que:compact_real} in the positive.
\end{exmp}

\subsection{Metrics on $G$ and $\Gamma$}

We denote by $m_k$ the Haar measure on $k$, normalised so that $m_k(\mathcal{O}_k) = 1$. This is related to the absolute value via $dm_k(ax) = |a|_k dm_k(x)$ for all $a\in k$. We consider the uniform norm $\|\cdot\|_k$ on $\operatorname{Mat}_{d\times d}(k)$, where $\|X\|_k := \max_{i,j} \{|X_{i,j}|_k \}$. This gives a natural left-invariant metric on $\SL_d(k)$, and by restriction on $G \leq \SL_d(k)$, via
\begin{equation*}
    d_G(A,B) := \log (1+\|A^{-1}B - \Id\|_k)
\end{equation*}
\begin{comment}
There is an alternative way to define a norm on $G$ via its Bruhat-Tits building. This is a locally finite simplicial complex of finite dimension on which $G$ acts continuously and properly, thus we can lift the combinatorial metric to a left-invariant metric on $G$, denoted $d_G'$. It is a standard fact (see for example \cite[\textsection 3.5]{LMR}) that these two metrics are coarse Lipschitz equivalent --- in particular for any neighbourhood $1\in U \subset G$ there is a constant $C = C(U) > 1$ such that for any $g \in G\backslash U$
\begin{equation*}
    C^{-1}d_{G}(1,g) \leq d'_{G}(1,g) \leq Cd_{G}(1,g).
\end{equation*}
Henceforth in the sequel we will only use $d_G$.
\end{comment}

Since $\Gamma$ is a uniform lattice we can easily relate word metrics on it to $d_G$. For any finite generating set of $\Gamma$ consider the associated word length metric $\ell_\Gamma$ on $\Gamma$. By the Milnor-Švarc lemma, $\ell_\Gamma$ and $d_G(e,\cdot)|_{\Gamma}$ are quasi-isometric.

\subsection{Tori and root spaces}

Let $A \cong (k^\times)^r$ be the maximal $k$-split torus of $G$. After fixing a uniformizer $\varpi$ of $k$ we have that $k^\times \cong \Z \times \mathcal{O}_k^\times$, and $\mathcal{O}_k^\times$, being a closed subset of $\mathcal{O}_k$, is compact. Thus the non-compactness of the torus comes from its $\Z^r$ factor, and from now on we will identify the two. That is, we set $A \cong \Z^r$ and call it the $k$-split torus. We have the Cartan decomposition $G=KAK$ with $K$ a maximal compact subgroup.

\myindent The metric $d_G$ when restricted to $A$ has the following useful property: for every $\Z$-basis $a_1,\ldots,a_r$ of $A$, there is some $C>0$ such that for any $a = a_1^{k_1}\cdots a_r^{k_r}$ we have
\begin{equation}
\label{eq:left_inv_metrics_are_qi}
    d_G(e,a) \leq (|k_1| d_G(e,a_1) + \cdots |k_r| d_G(e,a_r)) \leq Cd_G(e,a) + C
\end{equation}
This follows from the fact that all left-invariant metrics on finitely generated groups are quasi-isometric.

\myindent We have the adjoint representation $\Ad:G\to \GL(\mathfrak{g})$, we call the common eigenspaces for $\Ad(A)$ the \emph{root spaces}. Let $\mathbb{X}_k(A) = \{\chi:A\to k^\times\}$ be the set of multiplicative characters over $k$, for each $\chi$ consider the corresponding eigenspace
\begin{equation*}
    \mathfrak{g}_{\chi} = \{X\in \mathfrak{g}\mid \Ad(a) X = \chi(a)X \text{ for all }a\in A\},
\end{equation*}
we have that $|\chi(a)|_k = |\varpi|_k^{\alpha_\chi(a)}$ for some linear functional $\alpha_\chi:A\to \Z$. Notice that the only character $\chi$ with $\chi(a) = 1$ for all $a\in A$ is the trivial one.

\myindent We recall some basic facts about root spaces, see for example \cite{Springer, Borel}. We consider the (finite) set
\begin{equation*}
    \Sigma = \Sigma(G,A,k) = \{\alpha_\chi \mid \chi \in \mathbb{X}_k(A), \ \alpha_\chi \neq 0, \ \mathfrak{g}_{\chi} \neq 0\},
\end{equation*}
we will drop the subscript $\chi$ from now on, and denote the root spaces by $\mathfrak{g}_\alpha$. We can write
\begin{equation}
\label{eq:root_space_decomp}
    \mathfrak{g}=\Lie(Z_{G}(A))\oplus \left(\bigoplus_{\alpha\in \Sigma}\mathfrak{g}_{\alpha} \right)
\end{equation}
where $\Lie(Z_{G}(A))$, the Lie algebra of the centraliser of $A$ in $G$, coincides with the weight 0 subspace. The derived subgroup of $Z_{G}(A)$ is the \emph{anisotropic kernel} of $G$; although this will not be useful for us, we remark that in our non-Archimedean setting the anisotropic kernel can only be of type $A_n$.

\myindent In fact $\Sigma$ is a root system in the real vector space $\mathfrak{a}^*:=\operatorname{Hom}_\Z(A,\Z)\otimes_\Z \R$, called the \emph{relative root system} of $G$ over $k$ (the space $\mathfrak{a}^*$ is simply the $\R$-valued linear functionals on $A$). Since $G$ is simple, $\Sigma$ is irreducible. The following results are all classical.

\begin{prop} 
\label{prop:root_facts}
    Let $\Sigma = \Sigma(G,A,k)$ be the relative root system for $G$ corresponding to the $k$-split torus $A$. Then 
    \begin{enumerate}[label = {(\roman*)}]
        \item The roots are defined over $\Z$;
        \item \label{prop:root_facts:commutator_relation} $[\mathfrak{g}_{\alpha},\mathfrak{g}_{\beta}] \subset \mathfrak{g}_{\alpha+\beta}$, where it is understood that $\mathfrak{g}_{\alpha+\beta}=0$ if $\alpha +\beta \not\in \Sigma$;
        \item For any two nonzero roots $\alpha,\beta \in \Sigma$ such that $\alpha = c\beta$, then $c \in \{\pm \frac{1}{2},\pm 1,\pm 2\}$;
        \item To each $\alpha \in \Sigma$ there corresponds a $k$-defined connected unipotent subgroup $U_\alpha$, with Lie algebra $\mathfrak{g}_{\alpha}$ if $2\alpha\notin \Sigma$, and $\mathfrak{g}_{\alpha}\oplus \mathfrak{g}_{2\alpha}$ otherwise. The $U_\alpha$ generate $G$.
    \end{enumerate}
\end{prop}

See \cite[Appendix A]{brown_zimmers_2020} for a complete classification of the possible $\Sigma$'s and their Dynkin diagrams.

\begin{rem}
    In Section~\ref{sect:Averaging} we will use the $U_\alpha$ and the combinatorics of $\Sigma$. In some cases there are simplifying assumptions which we note now, although the argument will run in full generality.
    \begin{enumerate}[label = {(\alph*)}]
        \item If $G$ is split then in fact the $\mathfrak{g}_{\alpha}$ are one dimensional for $\alpha \neq 0$;
        \item If $\Sigma$ is not of type $BC_n$, then $\Sigma$ is \emph{reduced}, meaning that in (iii) we have $c\in \{\pm 1\}$. In particular (using (ii)) it is clear that for any $\alpha\in \Sigma$ we have that $[\mathfrak{g}_\alpha,\mathfrak{g}_\alpha] = \{0\}$. For $\Sigma$ of type $BC_n$ there are $n$ positive roots $e_i$ such that $2e_i$ is also a root.
    \end{enumerate}
\end{rem}

For any root system $\Sigma$ on a real vector space $V$ we have the following essential definition.

\begin{defn}
    A subset $\Delta \subset \Sigma$ is called a \emph{basis (or set of simple roots) of $\Sigma$} if $\Delta$ is a basis of $V$ as a vector space, and for any $\beta = \sum_{\alpha \in \Delta} c_i\alpha_i \in \Sigma$, all the $c_i$ are integers of the same sign.
\end{defn}

We collect some fundamental facts about bases.

\begin{prop}
\label{prop:bases_of_root_systems}
    Let $\Sigma$ be an irreducible root system on a real vector space $V$. Then
    \begin{enumerate}[label=(\alph*),ref=(\alph*)]
        \item Bases of $\Sigma$ are in one to one correspondence with the Weyl chambers;
        \item The Weyl group $W = W(\Sigma)$ acts irreducibly on $V$ and simply transitively on the set of bases;
        \item \label{prop:bases_of_root_systems:not_in_span}
        Let $0 \neq \lambda \in V$. Then there is some basis $\Delta = \{\alpha_1,\ldots,\alpha_r\}$ corresponding to the standard Dynkin diagram (as drawn in \cite[Appendix A]{brown_zimmers_2020}) such that $\lambda$ is not in the span of $\{\alpha_2,\ldots,\alpha_r\}$.
    \end{enumerate}
\end{prop}

\begin{proof}[Proof of (c)]
    Since $W$ acts irreducibly on $V$, for any proper hyperplane $V' \leq V$ there is some $w\in W$ such that $w(\lambda) \notin V'$.
     Let $\Delta = \{\alpha_1,\ldots,\alpha_r\}$ be any basis with the ordering coming from the standard Dynkin diagram and choose $w$ such that $w(\lambda)\notin \operatorname{span}\{\alpha_2,\ldots,\alpha_r\}$. If we denote $\alpha_i' = w^{-1}(\alpha_i)$ and set $\Delta' = \{\alpha_1',\ldots,\alpha_r'\}$ we have that $\lambda \notin \operatorname{span}\{\alpha_2',\ldots,\alpha_r'\}$. Since the action of $W$ preserves lengths and angles, the Dynkin diagram for $\Delta'$ has the same labelling as that of $\Delta$.
\end{proof}

We will use Proposition~\ref{prop:bases_of_root_systems}~(c) for some $\R$-valued linear functional $\lambda$ on $A$ (i.e. $\lambda \in \mathfrak{a}^*$), to show that there exists a basis $\Delta$ and $a\in A$ such that $\lambda(a) \neq 0$ but $\alpha_j(a) = 0$ for $j = 2,\ldots,r$.
    
    \section{Facts from smooth dynamics}
\label{sect:smooth_dynamics}

We recall some standard facts from smooth dynamics.

\subsection{Oseledets' Theorem}

Let $X$ be a compact metric space with a continuous left $G$-action. A measurable function $\mathcal{C}:G\times X \to \GL(m,\R)$ is called a \emph{linear cocycle} if 
\begin{equation*}
    \mathcal{C}(gh,x) = \mathcal{C}(g,hx)\mathcal{C}(h,x)
\end{equation*}
for all $g,h\in G$ and almost every $x\in X$.

Then $\mathcal{C}$ defines an action of $G$ by vector bundle automorphisms on the trivial bundle $\mathcal{E}:=X\times \R^m$ given by $g\cdot (x,v):= (gx,\mathcal{C}(g,x)v)$. This clearly projects to the $G$-action on $X$.

We will always assume that our cocycle is bounded; that is for every compact $K\subset G$ we have
\begin{equation*}
    \sup_{(g,x)\in K\times X}\|\mathcal{C}(g,x)\| < \infty
\end{equation*}

The following is a special case of the high-rank Oseledets' Theorem found in {\cite[Theorem 2.4]{brown_smooth_2023}}.

\begin{theo}\label{theo:Oseledets}
    Consider $H \cong \Z^k$ or $\R^k$ with a norm $|\cdot|$, and $X$ a compact metric space with a left $H$-action. Let $\mu$ be an $H$-invariant and $H$-ergodic probability measure, and $\mathcal{C}:H\times X \to \GL(m,\R)$ a bounded linear cocycle over this action. For each $x\in X$ let $E(x) \cong \R^m$ be the fiber of $X\times \R^m \to X$ over $x$. Then there are
    \begin{enumerate}[label = {(\roman*)}]
        \item an $H$-invariant subset $\Lambda_0\subset X$ with $\mu(\Lambda_0) = 1$;
        \item linear functionals $\lambda_i:H \to \R$ for $1 \leq i \leq p$ for some $1\leq p\leq m$, called \textbf{Lyapunov functionals};
        \item for every $x\in \Lambda_0$ a splitting $E(x) = \bigoplus_{i=1}^p E_i(x)$ into families of mutually transverse sub-bundles that vary measurably in $x$;
    \end{enumerate}
    such that
    \begin{enumerate}[label = {(\alph*)}]
        \item $\mathcal{C}(h,x)E_i(x) = E_i(hx)$ for every $x\in \Lambda_0$, and
        \item For all $x\in \Lambda_0$ and $v\in E_i(x)\backslash\{0\}$, 
        $$\lim_{|h|\to \infty}\frac{\log|\mathcal{C}(h,x)(v)|-\lambda_i(h)}{|h|} = 0.$$
    \end{enumerate}
\end{theo}

\begin{rem}
    In particular this implies the standard result that for any $h\in H$ and $v\in E_i(x)\backslash \{0\}$,
    \begin{equation}\label{eq:expansion}
        \lim_{k\to \pm\infty}\frac{\log|\mathcal{C}(h^k,x)(v)|}{k} = \lambda_i(h).
    \end{equation}
\end{rem}

If $\mu$ is $H$-invariant but not necessarily $H$-ergodic, then Theorem~\ref{theo:Oseledets} holds on each $H$-ergodic component of $\mu$. Even more is true; the data in Theorem~\ref{theo:Oseledets}, including the number of exponents $p(x)$, the subspaces $E_i(x)$, and the linear functionals vary measurably in $X$, see \cite[section 3.6.1]{BarreiraPesin2007}.

\subsection{Top Lyapunov exponent}

For a particular $a\in G$ we are often concerned with the fastest rate of expansion, in the notation of Theorem~\ref{theo:Oseledets} with $H = \langle a\rangle$ this is $\lambda_+(a,\mu,\mathcal{C}):= \max_i\lambda_i(a)$. For more general $H $ of higher dimension there is no reason for the measure $\mu$ to be ergodic with respect to $a$, so we define the \emph{average top Lyapunov exponent}
\begin{equation}
    \lambda_+(a,\mu,\mathcal{C}) = \lim_n \frac{1}{n}\int \log \|\mathcal{C}(a^n,x)\|_{\operatorname{op}}\, d\mu = \inf_n \frac{1}{n}\int \log \|\mathcal{C}(a^n,x)\|_{\operatorname{op}}\, d\mu, \label{eq:average_top_lyap}
\end{equation}
where we used subadditivity in the second equality. We note that if $\mu$ is in fact $a$-ergodic then by the Kingman subadditive ergodic theorem we have a.e. pointwise convergence of $\frac{1}{n}\log\|\mathcal{C}(a^n,x)\|_{\operatorname{op}}$ to $\lambda_+(a,\mu,\mathcal{C})$ as $n\to \infty$, recovering part of \eqref{eq:expansion}.

From the definition \eqref{eq:average_top_lyap} it is clear that if $\mathcal{C}(a^n,\cdot)$ is continuous for all $n >0$ then $\mu \mapsto \lambda_+(a,\mu,\mathcal{C})$ is upper semicontinuous on the space of $a$-invariant Borel probability measures, equipped with the weak-$*$ topology. 

\subsection{Averaging measures over amenable groups}

Given a locally compact amenable group $H$ with Haar measure $m_H$, a bounded measurable set $F \subset H$ of positive Haar measure, and a probability measure $\mu$ on $X$, we define the probability measure $F*\mu$ via 
\begin{equation*}
    \int_X f \, d(F\ast\mu) := \frac{1}{m_H(F)}\int_F \int_{X}f(h^{-1}x) \,d\mu(x) \, d m_H(h)
\end{equation*}
for any $f\in C(X)$.

The following is \cite[Lemma 4.2]{brown_zimmers_2020}, which we state without proof.

\begin{lem}
\label{prop:Top_Lyapunov_amenable_averaging}
    For some $a\in G$ let $H \leq Z_G(a)$ be amenable with Haar measure $m_H$ and $F_k$ a F\o lner sequence of precompact sets in $H$. Suppose that $\mathcal{C}$ is a linear cocycle over $X$ such that $x \mapsto \mathcal{C}(a^n,x)$ is continuous for all $n>0$. Let $\mu'$ be any weak-$*$ accumulation point of the $F_k*\mu$. Then
    \begin{enumerate}[label = {(\alph*)},ref=(\alph*)]
        \item \label{prop:Top_Lyapunov_amenable_averaging:center_invariance} $\mu'$ is $a$-invariant and $H$-invariant;
        \item \label{prop:top_Lyapunov_amenable_averaging:Amenable_average} $\lambda_+(a,\mu',\mathcal{C}) \geq \lambda_+(a,\mu,\mathcal{C})$. 
    \end{enumerate}
\end{lem}

\subsection{The suspension construction}
\label{sect:Suspension_construction}

As indicated in Section~\ref{sect:proof_sketch} we extend the $\Gamma$-action on $M$ to an action of $G$ on the suspension space. In the real setting this space is naturally a compact manifold of dimension $\dim M + \dim G$. In out $p$-adic setting, it will only be a continuous bundle over $G/\Gamma$.

\myindent We first recall the standard construction of the suspension space $M^\alpha$. Consider the right $\Gamma$-action on $G\times M$ given by $(x,y)\cdot \gamma := (x\gamma,\alpha(\gamma^{-1})y)$ and the left $G$-action $g\cdot(x,y) = (gx,y)$. These commute, and so the $G$-action descends to one on $M^\alpha:= (G\times M)/\Gamma$. There is a natural $G$-equivariant projection $\pi:M^\alpha \to G/\Gamma$, giving $M^\alpha$ the structure of a bundle over $G/\Gamma$, where the fibers are isomorphic to $M$, in particular $M^\alpha$ is compact. We give a more useful description.

Fix a compact open fundamental domain $\mathcal{F}$ for $\Gamma$ in $G$ (which exists in the totally disconnected setting), and consider the associated \emph{return cocycle}
\begin{equation*}
    \mathcal{R}_\Gamma:G \times \mathcal{F} \to \Gamma : (g,x) \mapsto \gamma_x(g) \text{ where } gx \in \mathcal{F}\gamma_x(g).
\end{equation*} 
Then there is a Borel trivialisation $\Phi:M^\alpha \to \mathcal{F}\times M$, which is $G$-equivariant when the latter space is equipped with the left $G$-action
\begin{equation}
\label{eq:alpha_tilde}
    g\cdot(x,y) := (gx\gamma_x(g)^{-1}, \alpha(\gamma_x(g))y). 
\end{equation}
Since $\mathcal{F}$ is open, for any fixed $g\in G$ the return cocycle $x\mapsto \gamma_x(g)$ is continuous.

\myindent From now we will write $\gamma(y) $ for $\alpha(\gamma)y$ where convenient, and we redefine $M^\alpha$ to be the $G$-space $\mathcal{F}\times M$ with the corresponding projection $\pi:M^\alpha\to G/\Gamma$. Consider the section $s:G/\Gamma \to \mathcal{F}$, we will often write $\gamma_x(g)$ for $\gamma_{s(x)}(g)$; this should be clear from context.

\myindent We can now define the \emph{fiberwise derivative} cocycle 
\begin{align*}
    \mathcal{D}:G\times M^\alpha \to \GL(\dim M,\R), \ \mathcal{D}_{(x,y)}(g) := D_{y}\left(\alpha(\gamma_{x}(g))\right)
\end{align*}
We check that this is a cocycle. Firstly, $\mathcal{R}_\Gamma$ satisfies the cocycle identity; indeed if we take any $h,g\in G$ and $x\in \mathcal{F}$, notice that
\begin{equation*}
    hgx = h(gx\gamma_x(g)^{-1})\gamma_x(g) \in \mathcal{F}\gamma_{gx\gamma_x(g)^{-1}}(h)\gamma_x(g)
\end{equation*}
where we used that $gx\gamma_x(g)^{-1}\in \mathcal{F}$ by definition; and so in particular
\begin{equation}
    \label{eq:return_cocyle}
    \gamma_x(hg) = \gamma_{gx\gamma_x(g)^{-1}}(h)\gamma_x(g).
\end{equation}
We can now verify immediately that
\begin{align*}
    \mathcal{D}_{(x,y)}(hg) & = D_{y}\left(\gamma_{x}(hg)\right) \\
    & = D_y(\gamma_{gx\gamma_x(g)^{-1}}(h)\gamma_x(g)) \\
    & = D_{\gamma_x(g)y}(\gamma_{gx\gamma_x(g)^{-1}}(h))\circ D_y(\gamma_x(g)) \\
    & = \mathcal{D}_{g\cdot (x,y)}(h)\mathcal{D}_{(x,y)}(g)
\end{align*}
by the definition \eqref{eq:alpha_tilde} of $g\cdot (x,y)$.

    \section{Positive Lyapunov exponent and the averaging argument}
\label{sect:Averaging}

In this section we prove Propositions~\ref{prop:exists_positive_lyapunov}~and~\ref{prop:average_to_haar}. Recall that $G = \mathbb{G}(k)$ is an algebraic group defined over $k$, a non-Archimedean local field of characteristic 0, and $\Gamma \leq G$ is a (necessarily uniform) lattice. In particular $M^\alpha$ is compact, and we do not need to worry about escape of mass.

\subsection{Proof of Proposition~\ref{prop:exists_positive_lyapunov}}
\label{sect:positive_lyapunov}
\begin{defn}
    We say that the $G$-action on $M^\alpha$ has \emph{uniform subexponential growth of fiberwise derivatives} if for every $\varepsilon > 0$ there is a constant $C_\varepsilon$ such that for any $g\in G$ we have
    \begin{equation*}
        \sup_{(x,y)\in M^\alpha} \|\mathcal{D}_{(x,y)} g\| \leq C_\varepsilon \mathrm{e}^{\varepsilon d(e,g)} 
    \end{equation*}
\end{defn} 

Since $\Gamma \leq G$ is uniform one can see that $\alpha$ has uniform subexponential growth of derivatives if and only if the $G$-action has uniform subexponential growth of fiberwise derivatives.

\myindent Let $C>0$ be the constant in \eqref{eq:left_inv_metrics_are_qi} relating the word length metric on $A$ to the metric induced from $G$. Suppose that $\alpha$ fails to have subexponential growth of derivatives, so for some $\varepsilon >0$ there are sequences $g_\ell \in G$, $(x_\ell,y_\ell)\in M^\alpha$, and $v_\ell \in T_{y_\ell}^1 M$ such that
\begin{equation*}      
    \|\mathcal{D}_{(x_\ell,y_\ell)}g_\ell(v_\ell)\| \geq \mathrm{e}^{\varepsilon(C+2) d(e,g_\ell)}
\end{equation*}
where $T^1M$ is the unit tangent bundle of $M$ with respect to the background Riemannian metric.

\myindent Recall that $K$ is the maximal compact subgroup in the Cartan decomposition $G=KAK$. By compactness, $d(e,k)$ is bounded for $k\in K$ and $d(e,x)$ is bounded for $x\in \mathcal{F}$, and hence $d(e,kx)$ is bounded over $kx\in K\mathcal{F}$. In particular, $K\mathcal{F} $ intersects finitely many $\mathcal{F}\gamma$'s. Since in addition $M$ is compact and our $\Gamma$-action is assumed $C^1$, we have that $\|\mathcal{D}|_{K\times M^\alpha}\|$ is uniformly bounded. So writing each $g_\ell = k_\ell s_\ell k_\ell'$, we have that for large enough $\ell$,
\begin{equation}
\label{eq:exponential_growth}
    \|\mathcal{D}_{(x_\ell,y_\ell)}s_\ell(v_\ell) \| \geq \mathrm{e}^{\varepsilon(C+1) d(e,s_\ell)}
\end{equation}

We say that an $a \in A$ \emph{witnesses the failure of uniform subexponential growth in fiberwise derivatives} if we can take $s_\ell = a^{n_\ell}$ in \eqref{eq:exponential_growth}, for some sequence $(n_\ell)$ with $|n_\ell| \to \infty$. 

\begin{claim}
    Suppose that the $G$-action fails to have uniform subexponential growth of fiberwise derivatives, and let $\{a_1,\ldots,a_r\}$ be any $\Z$-basis for $A$. Then some $a_i$ witnesses the failure of uniform subexponential growth in fiberwise derivatives.
\end{claim}

Indeed, suppose for contradiction that for our choice of $\varepsilon$ there exists some $C'>0$ such that for each $i$
\begin{equation*}
    \sup_{(x,y)\in M^\alpha}\|\mathcal{D}_{(x,y)} a_i^m\| \leq C' \mathrm{e}^{\varepsilon d(e,a_i^m)}.
\end{equation*}

Then for any $s = a_1^{m_1}\cdots a_r^{m_r}\in A$ and any $(x_0,y_0) \in M^\alpha$,
\begin{align*}
    \|\mathcal{D}_{(x_0,y_0)} s\| & \leq \sup_{(x,y) \in M^\alpha} \|\mathcal{D}_{(x,y)}a_1^{m_1}\|\ \cdots \sup_{(x,y) \in M^\alpha} \|\mathcal{D}_{(x,y)}a_{r}^{m_r}\| \\
    & \leq (C')^r \mathrm{e}^{\varepsilon (d(e,a_1^{m_1}) + \cdots + d(e,a_{r}^{m_r}))} \\
    & \leq C''\mathrm{e}^{\varepsilon Cd(e,s)}.
\end{align*}

This clearly contradicts \eqref{eq:exponential_growth} for large enough $d(e,s_\ell)$. 

\myindent Set $a=a_i$ to be the basis element that witnesses the failure of uniform subexponential growth of derivatives, by replacing $a$ with $a^{-1}$ if needed we can suppose that $n_\ell \to \infty$. That is, for some $\varepsilon >0$ we have that
\begin{equation*}
    \|\mathcal{D}_{(x_\ell,y_\ell)}a^{n_\ell}(v_\ell)\|\geq \mathrm{e}^{\varepsilon d(e,a^{n_\ell})}.
\end{equation*}
Once such an element exists, we use it to construct a measure with a positive top Lyapunov exponent. This is standard, for example see \cite[section 5]{brown_zimmers_2020}; we choose to include it for the readers' convenience.

\myindent Consider the space $T^1M^\alpha := \mathcal{F}\times T^1M$, this is a bundle over $M^\alpha$ where the fibers are identified with spheres of dimension $\dim M-1$. There is a natural way to induce a $G$-action on $T^1M^\alpha$, given by
\begin{equation*}
    T^1g(x,v) := \left(gx\gamma_x(g)^{-1}, \dfrac{\mathcal{D}_{(x,y)} g (v)}{\|\mathcal{D}_{(x,y)} g (v)\|}\right)
\end{equation*}
where it is understood that $y\in M$ is the basepoint of $v\in T^1M$. For each $\ell$, consider the Birkhoff measure on $T^1M^\alpha$ given by
\begin{equation*}
    \nu_\ell := \dfrac{1}{n_\ell}\sum_{j=0}^{n_\ell-1}\delta_{T^1a^j(x_\ell,v_\ell)}.  
\end{equation*}

We define, following the terminology in \cite{Cantat2017Bourbaki1136}, the \emph{logarithmic dilation} function of $a$ on $T^1M^\alpha$ via
\begin{equation*}
    \dil(x,v) := \log \|\mathcal{D}_{(x,y)} a(v)\|
\end{equation*}
and notice that since $\mathcal{D}$ is a cocycle, we have for any $n$ that
\begin{equation*}
    \log \|\mathcal{D}_{(x,y)} a^n(v)\| = \sum_{j=0}^{n-1}\dil(T^1 a^j(x,v)) 
\end{equation*}

By the construction of the $\nu_\ell$ we have that 
\begin{equation*}
    \int_{T^1M^\alpha} \dil(x,v) \, d\nu_\ell \geq \varepsilon.
\end{equation*}
Since $T^1M^\alpha$ is compact, we can take weak-$*$ limits.

\begin{claim}
    Any weak-$*$ limit $\nu$ of $(\nu_\ell)$ satisfies the following two properties.
    \begin{enumerate}[label = {(\alph*)}]
        \item $\nu$ is $T^1a$-invariant;
        \item $\int_{T^1M^\alpha} \dil(x,v) \, d\nu \geq \varepsilon > 0$.
    \end{enumerate}
\end{claim}

\begin{proof}
    If $\phi:T^1M^\alpha \to \R$ is any continuous test function, then 
    \begin{equation*}
        \lim_{\ell\to \infty} \Big|\int \phi \, d\nu_\ell - \int \phi \circ T^1a \, d\nu_\ell \Big|\leq \lim_{\ell \to \infty} \dfrac{2\|\phi\|_{C^0}}{n_\ell} = 0
    \end{equation*}
    and so (a) follows. Conclusion (b) follows from the definition of weak-$*$ convergence.
\end{proof}

We may readily replace $\nu$ by an $a$-ergodic component, which we will also denote $\nu$, that satisfies the inequality in (b). By the Birkhoff Ergodic Theorem we have for $\nu$-a.e. $(x,v) \in T^1M^\alpha$ that
\begin{align*}
    \lim_{n\to \infty}\frac{1}{n}\log \|\mathcal{D}_{(x,y)}a^n(v)\| = \lim_{n\to \infty} \frac{1}{n}\sum_{j=0}^{n-1} \dil(T^1a^j(x,v))  = \int \dil \, d\nu \geq \varepsilon.
\end{align*}
Let $\mu = P_*\nu$, where $P:T^1M^\alpha \to M^\alpha$ is the natural projection. Then $\mu$ is an $a$-invariant, $a$-ergodic measure on $M^\alpha$. Consider also the family of conditional probability measures $\nu_{(x,y)}$ of $\nu$ for the partition into fibers of $P$.

\begin{proof}[Proof of Proposition~\ref{prop:exists_positive_lyapunov}]
    We can calculate
\begin{align*}
    \lambda_{+}(a,\mu,\mathcal{D}) & = \lim_{n\to \infty}\dfrac{1}{n}\int\log \|\mathcal{D}_{(x,y)} a^n\|_{\operatorname{op}} \, d\mu(x,y) \\
    & = \lim_{n\to \infty}\dfrac{1}{n}\int\sup_{v\in T^1_yM}\log \|\mathcal{D}_{(x,y)} a^n(v)\| \, d\mu(x,y) \\
    & =  \lim_{n\to \infty}\dfrac{1}{n}\int \sup_{v \in T^1_yM}\sum_{j=0}^{n-1} \dil (T^1a^j(x,v)) \, d\mu(x,y) \\
    & \geq \lim_{n\to \infty}\int \left(\int\dfrac{1}{n} \sum_{j=0}^{n-1} \dil (T^1a^j(x,v)) \, d\nu_{(x,y)}(v)\right) \, d\mu(x,y) \\
    & = \int \dil(x,v) \, d\nu \geq \varepsilon > 0.
\end{align*}

Hence there is a measure $\mu$ with a strictly positive Lyapunov exponent for $a$. The result follows by averaging with respect to a F\o lner sequence in $A$ using Lemma~\ref{prop:Top_Lyapunov_amenable_averaging}, and passing to an ergodic component.
\end{proof}

\subsection{Facts from Unipotent Dynamics}

We recall some facts from the dynamics of unipotent flows that will be vital for us, starting with some definitions. The image $U$ of a homomorphism $u:k \to G$ is a \emph{one-parameter unipotent subgroup} if every element in its image is unipotent; that is, its only eigenvalue is 1.

\begin{defn}
    A $U$-invariant measure $\mu$ on $G/\Gamma$ is \emph{homogeneous} if there exists some closed subgroup $U\leq L\leq G$ such that $\mu$ is $L$-invariant, and a point $x\in G/\Gamma$ such that $L\cdot x$ is closed and supports $\mu$. Such a measure will be denoted by $m_L$. Similarly, a $U$-invariant subset $B\subset G/\Gamma$ is \emph{homogeneous} if it is closed and of the form $L\cdot x$ for a closed subgroup $U\leq L\leq G$.
\end{defn}

For any $T>0$ we consider the subset of $k$ given by
\begin{equation*}
    F_k(T) := \{x\in k \mid |x|_k \leq T\}.
\end{equation*}
As $T \to \infty$ these form a F\o lner sequence in $k$ of compact-open subgroups. 

\begin{defn} \label{defn:generic}
    A point $x \in G/\Gamma$ is \emph{generic (for $U$)} if there is a closed subgroup $U\leq L \leq G$ such that $\overline{U\cdot x} = L\cdot x$ is homogeneous, and
    \begin{equation*}
        \frac{1}{m_k(F_k(T))}\int_{F_k(T)}f(u(t)x) \, dm_k(t) \longrightarrow \int_{G/\Gamma} f \, dm_L \quad \text{as } T\to \infty.
    \end{equation*}
    Recall that $m_k$ is the Haar measure on $k$, and $m_L$ the homogeneous measure on $L\cdot x$.
\end{defn}

In general the group $L$ and the measure $m_L$ depend on both $x$ and $U$. For a point $x$ that is generic for $U$ we will denote the measure that arises in the definition by $m_x$, with the $U$-dependence kept implicit. We now state some cases of Ratner's Theorems. For groups over non-Archimedean local fields of characteristic 0 they are proved independently in \cite{ratner_raghunathans_1995} and \cite{margulis_invariant_1994}, both building on the methods developed in \cite{Ratner1, Ratner2, Ratner3}.

\begin{theo}[Ratner's Theorems]
\label{theo:Ratner}
    Let $U$ be a one-parameter unipotent subgroup. Then
    \begin{enumerate}[label=(\alph*),ref=(\alph*)]
        \item Every $U$-invariant and $U$-ergodic probability measure $\mu$ is homogeneous (\cite[Theorem 1]{ratner_raghunathans_1995}, \cite[Theorem 2]{margulis_invariant_1994}). 
        \item For every $x\in G/\Gamma$ the orbit closure $\overline{\{u \cdot x :u\in U\}}$ is homogeneous (\cite[Theorem 2]{ratner_raghunathans_1995}, \cite[Theorem 11.1]{margulis_invariant_1994}).
        \item \label{theo:Ratner:equidistribution} Every point in $G/\Gamma$ is generic for $U$ (\cite[Theorem 3]{ratner_raghunathans_1995}, \cite[Theorem 11.2]{margulis_invariant_1994}).
        \end{enumerate}
\end{theo}

The following is a special case of \cite[Corollary 2]{ratner_raghunathans_1995}, applied to one-parameter unipotent subgroups that generate the root subgroups.

\begin{cor}
\label{cor:Ratner_symmetry_of_entropy} 
Let $\beta \in \Sigma$ be a relative root, and let $\mu$ be a measure on $G/\Gamma$ invariant under $\langle U_\beta,A\rangle$. Then $\mu$ is invariant under $U_{-\beta}$. 
\end{cor}

Following \cite{brown_zimmers_2020}, for any measure $\mu$ on $G/\Gamma$ we define the measure
\begin{equation*}
    U \ast \mu := \int m_{x} \, d\mu(x).
\end{equation*}
We remark that by dominated convergence we have that
\begin{equation}
\label{eq:dominated_convergence}
    F_k(T) \ast \mu \longrightarrow U \ast \mu \quad \text{as } T\to \infty.
\end{equation}
We stress that in particular $U\ast\mu$ is the only weak-$*$ limit of the $F_k(T)\ast\mu$ as $T \to \infty$. Using this observation, it is possible to show that $U\ast\mu$ maintains properties of $\mu$ with respect to an $A$ that normalises it (as will always be the case with root subgroups by definition), which we recall below. This is \cite[Proposition 6.2]{brown_zimmers_2020}, we omit the proof as it is essentially identical.

\begin{prop}
\label{prop:averaging_facts}
    Let $U$ be a unipotent one-parameter subgroup normalised by $A$ and  $\mu$ a Borel probability measure on $G/\Gamma$. Then
    \begin{itemize}
        \item if $\mu$ is $A$-invariant then $U\ast \mu$ is $\langle A,U\rangle$-invariant;
        \item if $\mu$ is $A$-invariant and $A$-ergodic then $U\ast \mu$ is $A$-ergodic. 
    \end{itemize} 
\end{prop}

\subsection{Averaging argument: first step}

Now we explain how to average measures on $G/\Gamma$ to obtain some extra invariance; this will go some ways to proving Proposition~\ref{prop:average_to_haar}. Let $\preccurlyeq$ be an order on $\mathfrak{a}^*$, $\Sigma_+$ the set of positive roots for $\preccurlyeq$, and $\Delta\subset\Sigma$ the corresponding basis. For any subset $\Delta' \subset \Delta$ we consider the subgroup $U_{\Delta'} $ generated by $\{U_\alpha:\alpha\in \Delta'\}$. This is a unipotent subgroup of $P = P_\Delta$, the Borel subgroup corresponding to $\Delta$.

Let $\Sigma'$ be the root system generated by $\Delta'$, we can order its $\preccurlyeq$-positive elements as $\Sigma'_+ = \{\beta_1,\ldots,\beta_\ell\}$, where $\beta_i\preccurlyeq \beta_j$ whenever $i\leq j$. 

Define the height function 
\begin{equation*}
    h:\Sigma_+'\to \N_0:\sum_{\alpha\in \Delta'}c_\alpha \alpha \mapsto \sum_{\alpha\in \Delta'} c_\alpha.
\end{equation*}
Observe that if $h(\beta) \leq h(\beta')$ then certainly $\beta\preccurlyeq \beta'$. The \emph{height filtration} $\{U_r\}$ for $r \geq 0$ by $U_r := \langle U_\beta:\beta\in \Sigma', h(\beta) \geq r\rangle$, where $\Sigma_r' := \{\beta\in \Sigma_+'\mid h(\beta)\geq r\}$. We have (by Proposition~\ref{prop:root_facts}~\ref{prop:root_facts:commutator_relation}) that
\begin{equation} \label{eq:derived_series}
    [U_r,U_r] \subset U_{2r}.
\end{equation}

We make the following observation.
\begin{lem}
\label{lem:commutator_invariance}
    Let $U,V$ be one-parameter unipotent subgroups, and $\nu$ a measure on $G/\Gamma$ that is both $V$ and $[U,V]$-invariant. Then $U*\nu$ is $U,V$-invariant.
\end{lem}

In the case that $[U,V] = 1$ this is part of Lemma~\ref{prop:Top_Lyapunov_amenable_averaging}~\ref{prop:Top_Lyapunov_amenable_averaging:center_invariance}. 

\begin{proof}
    Let $u\in U$, and consider any $v\in V$. Then
\begin{equation*}
    v_*(u_*\nu) = u_*v_*([v^{-1},u^{-1}]_*\nu) = u_*\nu.
\end{equation*}
Thus $u_*\nu$ remains $V$-invariant, hence so is $F_k(T)\ast\nu$. Passing to the limit, the result follows.
\end{proof}

\myindent This allows us to make the following definition. Let $\beta \in \Sigma$ be a root, and pick one dimensional unipotent subgroups $U_\beta^1, \ldots, U_\beta^s$ that generate $U_\beta$. Let $\nu$ be an $A$-invariant and $A$-ergodic measure; if $2\beta \in \Sigma$ then assume in addition that $\nu$ is $U_{2\beta}$-invariant. Then define 
\begin{equation*}
    U_\beta * \nu := U_\beta^1*\cdots \ast U_\beta^s *\nu.
\end{equation*}
By Lemma~\ref{lem:commutator_invariance} and Proposition~\ref{prop:averaging_facts} applied repeatedly it follows that $U_\beta\ast\nu$ is invariant under $\langle A, U_\beta\rangle$, and $A$-ergodic. The $U_\beta^i$ we pick and the order we pick them in will be of no material consequence.

\begin{lem}
\label{lem:solvalble_averaging}
    Let $\nu$ be an $A$-invariant and ergodic measure on $G/\Gamma$. Let $\Delta' \subset \Delta$ be a subset of a basis of $\Sigma$, with corresponding height function $h:\Sigma'_+ \to \N$. Define recursively the measures
    \begin{equation*}
        \nu_{\ell} = U_{\beta_\ell} * \nu \quad \text{and} \quad \nu_k = U_{\beta_k} * \nu_{k+1} \text{ for } k = \ell-1,\ldots,1.
    \end{equation*}
    Then $\nu_1$ is $\langle A, U_{\Delta'}\rangle $-invariant and $A$-ergodic.
\end{lem}

\begin{proof}
    By Proposition~\ref{prop:averaging_facts} each $\nu_k$ is $A$-invariant and $A$-ergodic. Since $\nu_{k+1}$ is $U_{\beta_j}$-invariant for every $j\geq k+1$, by Lemma~\ref{lem:commutator_invariance} and the fact that $h([\beta_k,\beta_j]) > h(\beta_k)$ we see that $\nu_k$ is $U_{\beta_j}$-invariant for $j\geq k$. By induction this holds for $k=1$. 
\end{proof}

\subsection{Averaging on \texorpdfstring{$M^\alpha$}{}}

We consider what happens when we average measures on $M^\alpha$ instead of $G/\Gamma$; in particular we lack Ratner's equidistribution statement Theorem~\ref{theo:Ratner}~\ref{theo:Ratner:equidistribution}. Let $\mu$ and $a$ be as in the conclusion of Proposition~\ref{prop:exists_positive_lyapunov}, and let $\lambda = \lambda_i(\cdot,\mu,\mathcal{D}):A\to \R$ be a Lyapunov exponent such that $\lambda_i(a,\mu,\mathcal{D}) = \lambda_+(a,\mu,\mathcal{D}) >0$. By Proposition~\ref{prop:bases_of_root_systems}~\ref{prop:bases_of_root_systems:not_in_span} we can find some basis $\Delta = \{\alpha_1,\ldots,\alpha_r\}$ of $\Sigma$ such that $\lambda \notin \operatorname{span} \{\alpha_2,\ldots,\alpha_r\}$. 

\myindent In particular, we can find some $a' \in A$ such that $\lambda_i(a',\mu,\mathcal{D}) > 0$ but $\alpha_j(a') = 0$ for all $j=2,\ldots,r$ (this is possible as the $\alpha_j$ are defined over $\Z$, and hence so are their kernels). Setting $\Delta' = \Delta\backslash \{\alpha_1\}$ we have that $U_{\Delta'}\leq Z_G(a')$. As before, for any root $\beta \in \Sigma$ we can consider one-parameter subgroups $U_\beta^1,\ldots,U_\beta^s$ that generate $U_\beta$, and consider the F\o lner sequence $\{F^i_k(T)\}$ in $U^i_\beta$. By a slight abuse of notation, we let $U_\beta^i \ast \mu$ be \emph{any} choice of weak-$*$ limit of the $F^i_k(T) \ast \mu$, and define $U_\beta \ast \mu = U_\beta^1\ast \cdots \ast U_\beta^s \ast \mu$. We note that by definition $\pi_*(U_\beta\ast\mu) = U_\beta \ast( \pi_*\mu)$.

\myindent As in Lemma~\ref{lem:solvalble_averaging} we can define recursively the measures
\begin{equation*}
    \mu_{\ell} = U_{\beta_\ell} * \mu \quad \text{and} \quad \mu_k = U_{\beta_k} * \mu_{k+1} \text{ for } k = \ell-1,\ldots,1.
\end{equation*}
We stress that (as we are on $M^\alpha$) we are choosing a weak-$*$ limit at every step, so this is not uniquely defined. In contrast after projecting to $G/\Gamma$ (and fixing the one-parameter groups $U_\beta^j$ that generate $U_\beta$) the measure $\pi_*(\mu_\ell)$ is uniquely defined by \eqref{eq:dominated_convergence}.

The measure $\mu_1$ that we obtain after this process satisfies:
\begin{enumerate}[label = {(\roman*)}]
    \item $\lambda_+(a',\mu_1,\mathcal{D})>0$ by Lemma~\ref{prop:Top_Lyapunov_amenable_averaging}~\ref{prop:top_Lyapunov_amenable_averaging:Amenable_average}, and
    \item $\pi_*(\mu_1)$ is $\langle A,U_{\Delta'}\rangle$-invariant and $A$-ergodic, by Lemma~\ref{lem:solvalble_averaging} applied to $\nu = \pi_*\mu$.
\end{enumerate}

Consider now a F\o lner sequence $F^A_i$ in $A$, and let $\mu_1'$ be an ergodic component of a weak-$*$ limit of $F^A_i\ast \mu_1$ as $i\to \infty$. Then $\mu_1'$ satisfies:
\begin{enumerate}[label = {(\roman*)}]
    \item $\lambda_+(a',\mu_1',\mathcal{D})>0$,
    \item $\pi_*(\mu_1') $ is $\langle A,U_{\Delta'}\rangle$-invariant and $A$-ergodic (since $\pi^*(\mu_1)$ is $A$-invariant and so clearly $\pi_*(\mu_1') = \pi_*(\mu_1)$), and
    \item $\mu_1'$ is $A$-invariant and $A$-ergodic.
\end{enumerate}

\subsection{Averaging argument: second step} We average over another group to complete the proof of Proposition~\ref{prop:average_to_haar}. Let $a'$ and $\mu_1'$ be as above. 

Now let $\lambda = \lambda_i(\cdot,\mu_1',\mathcal{D}):A\to \R$ be such that $\lambda_i(a',\mu_1',\mathcal{D}) = \lambda_+(a',\mu_1',\mathcal{D})$. There are two possible cases to consider: $\lambda$ is either not proportional to $\alpha_1$, or it is. In each case we pick a different element of $A$ and averaging subgroup that commutes with it.

\myindent Suppose first that $\lambda$ is not proportional to $\alpha_1$. Then we can pick some $a''\in A$ such that $\alpha_1(a'') = 0$ but $\lambda(a'') >0$. Consider then a weak-$*$ limit $\mu_0 = U_{\alpha_1}\ast \mu_1'$. By Proposition~\ref{prop:averaging_facts} $\pi_*(\mu_0)$ is $\langle A,U_{\alpha_1}\rangle$-invariant and $A$-ergodic. Since $[U_{\alpha_1},U_{-\alpha_k}] = 1$ for $k =2, \ldots,r$ we have that $\pi_*(\mu_0)$ is $U_{-\alpha_k}$-invariant for $k=2,\ldots,r$, and by Corollary~\ref{cor:Ratner_symmetry_of_entropy} it is invariant under $U_{\pm\alpha_k}$ for all $k$. In particular, it is the Haar measure on $G/\Gamma$.

\myindent Suppose instead that $\lambda$ is proportional to $\alpha_1$. Let $\delta = \sum m_i \alpha_i$ be the longest root, that is the root of greatest height in $\Sigma_+$ (alternatively, it is the maximal root with respect to the ordering $\preccurlyeq$). Let 
\begin{equation*}
    \delta' = \begin{cases}
        \delta & \text{if } m_1=1; \\
        \delta-\alpha_1 & \text{if } m_1=2.
    \end{cases}
\end{equation*}
A table of the $\delta'$ can be found in the appendix of \cite{brown_zimmers_2020}. We observe that by this choice $[U_{\alpha_k},U_{\delta'}]=1$ for $k=2,\ldots,r$. Furthermore, there is a chain of roots $\delta_i\in \Sigma$ such that
\begin{equation}\label{eq:chain_of_roots} 
    \delta_0 = \alpha_1, \quad \delta_i = \delta_{i-1} + \alpha_{k_i}\text{ for some } k_i \in \{2,\ldots, r\}, \quad \text{and} \quad \delta_{s} = \delta',
\end{equation}
where $s = h(\delta')-1$. Now pick an $a''$ such that $\delta'(a'') = 0$ but $\lambda(a'') >0$.

\myindent We choose $\mu_0$ to be any weak-$*$ limit of $U_{\delta'}\ast\mu_1'$. By Proposition~\ref{prop:averaging_facts} we have that $\pi_*(\mu_0)$ is $\langle A,U_{\delta'}\rangle$-invariant and $A$-ergodic. In addition, it is $U_{\alpha_k}$-invariant for $k=2,\ldots,r$ since these groups commute with $U_{\delta'}$, and by Corollary~\ref{cor:Ratner_symmetry_of_entropy} also $U_{-\alpha_k}$-invariant for $k=2,\ldots,r$. We then also have $[U_{\delta'},U_{-\alpha_{k_{s}}}] = U_{\delta_{s-1}}$-invariance, and we can descend down the chain in \eqref{eq:chain_of_roots} to obtain $U_{\alpha_1}$-invariance. By Corollary~\ref{cor:Ratner_symmetry_of_entropy} again we get that $\pi_*(\mu_0)$ is the Haar measure on $G/\Gamma$.

\myindent Summarising the two cases, we have that $\mu_0$ satisfies:
\begin{enumerate}[label = {(\roman*)}]
    \item $\lambda_+(a'',\mu_0,\mathcal{D})>0$ by Lemma~\ref{prop:Top_Lyapunov_amenable_averaging}~\ref{prop:top_Lyapunov_amenable_averaging:Amenable_average}, and
    \item $\pi_*(\mu_0) =m_{G/\Gamma}$ is the Haar measure.
\end{enumerate}

Finally we can let $\mu_0'$ be an ergodic component of a weak-$*$ limit of $F_i^A\ast\mu_0$ as $i\to \infty$. By Lemma~\ref{prop:Top_Lyapunov_amenable_averaging}~\ref{prop:top_Lyapunov_amenable_averaging:Amenable_average}, and the fact that $\pi_*(\mu_0)$ is $A$-invariant and so $\pi_*(\mu_0') = \pi_*(\mu_0) = m_{G/\Gamma}$, we have completed the proof of Proposition~\ref{prop:average_to_haar}. \qed

\begin{rem}
    We remark that the averaging argument can be made far more general, provided that one has a sufficiently nice cocycle (specifically one needs Lemma~\ref{prop:Top_Lyapunov_amenable_averaging} to apply). That is, it can be made to work for any $S$-arithmetic group as defined in Section~\ref{sect:S-arithmetic}, and one does not need to make any assumptions on the dimension of the manifold or the rank of the factors. We chose not to write this argument in full generality, just focusing on the setup needed for the proof of Theorem~\ref{theo:simple_p_adic_case}.
\end{rem}

     \section{Proof of Theorem~\ref{theo:G_invariance}: The invariance principle}
\label{sect:invariance}

In this section we prove Theorem~\ref{theo:G_invariance}, and thus complete the proof of Theorem~\ref{theo:simple_p_adic_case} as outlined in Section~\ref{sect:proof_sketch}. In fact, we will prove a result more general than Theorem~\ref{theo:G_invariance}, allowing for all $S$-arithmetic groups. We first recall a standard definition.

\begin{defn}
    For an element $1\neq a\in G$, the \emph{stable horospherical subgroup} is
\begin{equation*}
    G_a^- := \{g:a^nga^{-n}\to e\mid n\to \infty\}
\end{equation*}
and the \emph{unstable horospherical subgroup} is  $G_a^+ := G_{a^{-1}}^-$. 
\end{defn}

\begin{theo}
    \label{theo:invariance_under_stanble_and_unstable_horosphericals}
    Let $S\subset \mathcal{V}$ be a finite set of places of $\mathbb{Q}$ and let $G$ be an $S$-arithmetic group and $\Gamma \leq G$ an irreducible uniform lattice. For each $p\in S$ let $A_p\leq \mathbb{G}_p(\mathbb{Q}_p)$ be a $\mathbb{Q}_p$-split torus (not necessarily maximal) of dimension $l_p \geq 0$. Let $\dim A := \sum_{p\in S} l_p$.
    
    \myindent Let $M$ be a compact manifold and $\alpha:\Gamma\to \Diff^1(M)$ a homomorphism, and construct the corresponding suspension space $M^\alpha$ and projection $\pi:M^\alpha \to G/\Gamma$. Let $\mu$ be any probability measure on $M^\alpha$ that is $A$-invariant and $A$-ergodic and such that $\pi_*(\mu) = m_{G/\Gamma}$ is the Haar measure.

    Suppose in addition that either
    \begin{enumerate}
        \item $\dim M < \dim A$, or
        \item $\dim M = \dim A$ and $\alpha(\Gamma)$ preserves a $C^1$-volume form.
    \end{enumerate}
    Then there is some non-trivial $1\neq a\in A$ such that $\mu$ is invariant under $\langle G_a^+,G_a^-\rangle$.
\end{theo}

\begin{proof}[Proof of Theorem~\ref{theo:G_invariance}]
    In the setting of Theorem~\ref{theo:G_invariance}, $A$ is a maximal $\mathbb{Q}_p$-split torus. Since $G$ is simple, for any $1\neq a\in A$ we have that $\langle G_a^-, G_a^+\rangle = G$ and so the result follows from Theorem~\ref{theo:invariance_under_stanble_and_unstable_horosphericals}.
\end{proof}

\subsection{Facts about Horospherical subgroups}

Letting $\mathfrak{g}^\pm$ be the Lie algebras of $G_a^\pm$ we have a decomposition $\mathfrak{g} = \mathfrak{g}_- \oplus \mathfrak{g}_0 \oplus \mathfrak{g}_+$, where $\mathfrak{g}_0 = \Lie(Z_G(a))$ as before. Notice that 
\begin{equation*}
    \mathfrak{g}_- = \bigoplus_{\alpha\in \Sigma: \alpha(a) < 0} \mathfrak{g}_\alpha
\end{equation*}
and similarly for $\mathfrak{g}_+$.

\myindent We know that $A\cong \R^{l_\infty} \times \prod_{p\in S\backslash \{\infty\}}\Z^{l_p}$. The $\R^{l_\infty}$-factor can potentially simplify some steps, but taking account of this will only introduce burdensome notation and will not improve the final result. As such, we simply pick a copy of $\Z^{l_\infty}\leq A_\infty$ and redefine $A := \Z^{l_\infty} \times \prod_{p\in S\backslash \{\infty\}}\Z^{l_p} \cong \Z^{\dim A}$. The Cartan decomposition is no longer true for this $A$, but this is irrelevant for what follows. We will not comment on this further.

We define the notation $A_\R := A \otimes_\Z\R$, and let $S(A_\R) = (A_\R\backslash\{0\})/\R_{>0}$ be the space of directed rays in $A_\R$ (which is canonically the unit sphere in $A_\R$, hence the notation). Consider a point $[v]\in S(A_\R)$, and any sequence $(a_j)\in A$ such that
\begin{enumerate}[label = {(\roman*)}]
    \item $[a_j] \to [v]$ as $j\to \infty$ in $S(A_\R)$;
    \item If $\alpha(v) = 0$ for some $\alpha \in \Sigma$, then $\alpha(a_j) = 0$ for all $j$.
\end{enumerate}
Then the subgroup $G_{a_j}^-$ is eventually constant, and is independent of the sequence $(a_j)$ we chose. We define $G_v^-$ to be this subgroup. Alternatively we can define it via $G_v^- := \langle U_\alpha:\alpha(v) < 0\rangle$. This group depends only on the smallest simplex containing $v$ in the Weyl complex associated to the root system $\Sigma$. For any $[v]\in S(A_\R)$ consider the subset $\mathcal{C}_v \subset A$ consisting of elements $a$ such that
\begin{enumerate}[label = {(\roman*)}]
    \item $\alpha(a) = 0$ whenever $\alpha(v) = 0$,
    \item $\alpha(a) > 0$ whenever $\alpha(v) > 0$, and
    \item $\frac{\alpha(a)}{|a|} < \frac{\alpha(v)}{2|v|}$ whenever $\alpha(v) < 0$.
\end{enumerate}
Notice that the smallest cone in $A_\R$ containing $C_v$ is given by the intersection of an open cone in $A_\R$ with the hyperplanes corresponding to $\alpha(v) = 0$, and in particular if $a\in \mathcal{C}_v$ then $a^n \in \mathcal{C}_v$ for all $n \geq 1$. We remark that we can alternatively define $G_v^- := G_a^-$ for any $a\in \mathcal{C}_v$.

\subsection{Subordinate partitions}
\label{sect:partition}

Let $(X,\mathcal{B}_X,\nu_X)$ be a Lebesgue space, we recall some necessary facts about partitions of it (see for example \cite[Appendix 1]{cornfeld_fomin_sinai_1982}). We say that two collections $\mathcal{C},\mathcal{C}'$ of measurable sets \emph{coincide (mod 0)} if for every $A\in \mathcal{C}$ there is some $A\in \mathcal{C}'$ such that $\nu_X(A\triangle A') = 0$ and vice versa.

\myindent For any partition $\xi$ of $X$ consisting of measurable sets we can associate a complete $\sigma$-algebra $\mathcal{B}_\xi \subset \mathcal{B}_X$, consisting of measurable sets which coincide mod 0 with unions of sets in $\xi$. Conversely to any sub-$\sigma$-algebra $\mathcal{A} \subset \mathcal{B}_X$ we can associate a partition $\xi(\mathcal{A})$ such that $\mathcal{B}_{\xi(\mathcal{A})}$ is equivalent to $\mathcal{A}$ mod 0.  It is not always the case that $\xi({\mathcal{B}_\xi})$ is equivalent mod $0$ to $\xi$, even in the case where all the $\xi$ are measurable sets; to guarantee this it suffices to assume that $\mathcal{B}(\xi)$ is countably generated. In this case we say that the partition $\xi$ is \emph{measurable}. We say that a function is \emph{$\xi$-measurable} if it is measurable with respect to $\mathcal{B}_\xi$. Given a partition $\xi$, for $x\in X$ we denote by $[x]_\xi$ the partition element containing $x$; if $\mathcal{B}_\xi$ is countably generated this is equivalently the atom of $x$ in $\mathcal{B}_\xi$. Given two partitions $\xi$, $\xi'$ we say that $\xi'$ is \emph{finer} than $\xi$, denoted $\xi \prec \xi' $, if every element of $\xi$ is a union of elements of $\xi'$ (mod 0), equivalently $\mathcal{B}_\xi \subset \mathcal{B}_{\xi'}$.

\myindent The following is an easy modification of a standard construction, see for example \cite[Proposition 7.37]{Pisa}. We remark that the Haar measure $m_{G/\Gamma}$ is ergodic for any (non-trivial) $a\in A$. 

\begin{prop}
\label{prop:subordinate_partition}
    Consider some $[v] \in S(A_\R)$. Let $\nu$ be an $A$-invariant probability measure on $G/\Gamma$, and assume in addition it is ergodic for every $a\in \mathcal{C}_v$. Given $U \leq G_v^-$ a closed $A$-normalised subgroup, there is a measurable partition $\xi$ of $G/\Gamma$ such that 
    \begin{enumerate}
        \item $\xi$ is \textbf{subordinate to $U$}; that is for $\nu$-a.e. $x$ there exists $\delta>0$ such that
        \begin{equation*}
            B_\delta^U\cdot x \subset [x]_\xi \subset B_{\delta^{-1}}^U\cdot x
        \end{equation*}
        \item $\xi$ is \textbf{$a$-decreasing} for every $a\in \mathcal{C}_v$, that is $a^{-1}\xi \prec \xi$.
    \end{enumerate}
    Here $B_r^U$ denotes the $r$-ball in $U$.

    Furthermore, it follows from the construction and the fact that $G/\Gamma$ is compact that there is some $\delta_0 > 0$ such that for $\nu$-a.e. $x$, $[x]_\xi \subset B_{\delta_0^{-1}}^U\cdot x$. In particular, the upper bound in $U$ on $[x]_\xi$ is uniform for $\nu$-a.e. $x$.
\end{prop}

\subsection{Dynamics on the partition}

Consider now the Haar measure $m_{G/\Gamma}$ on $G/\Gamma$ and for some $[v] \in S(A_\R)$ the group $U = G_v^-$ and the partition $\xi$ given by Proposition~\ref{prop:subordinate_partition}. Let $\psi:G/\Gamma \to G/\Gamma$ be a measurable selection map (see for example \cite[Lemma 4.6]{avila_extremal_2010}); that is $\psi$ is defined such that $\psi(x) \in [x]_\xi$ for a.e. $x$, and $\psi$ is constant on atoms of $\xi$. Let $s:G/\Gamma \to G$ be the section of the projection map $\pi:G\to G/\Gamma$ corresponding to the fundamental domain $\mathcal{F}$. Let $\overline{\psi} = s\circ \psi$.

By Poincar\'{e} recurrence it follows that for a generic $x$ the map $u \mapsto u\cdot x$ is injective, so let $U_x $ be the inverse image of $[x]_\xi$ under the orbit map $u\mapsto u\cdot \psi(x)$. Define $\xi_1(x) = U_x\overline{\psi}(x)$, and notice that
\begin{enumerate}
    \item $\pi(\xi_1(x)) = [x]_\xi$;
    \item $\xi_1(x) \cap \mathcal{F} \neq \emptyset$.
\end{enumerate}

Given a generic $x\in G/\Gamma$, let $u_x$ be such that $\psi(x) = u_x x$. For any $a\in \mathcal{C}_v$, define the diffeomorphism $F_x^a:M\to M$ to be
\begin{equation*}
    F_x^a = \gamma_{\psi(x)}(u_{a\psi(x)}a)
\end{equation*}
and consider the map $F^a:G/\Gamma\times M \to G/\Gamma \times M: (x,y) \mapsto (a x\gamma_x(a)^{-1}, F_x^a(y))$. Recall that $M^\alpha = G/\Gamma\times M$ and let $\Phi:G/\Gamma \times M \to G/\Gamma \times M$ be the map $(x,y) \mapsto (x,\gamma_x(u_{x})y)$. We denote $\mu' := \Phi_*\mu$. The following is an adaptation of \cite[Claim 1]{brown_c1_2022}.

\begin{prop}
\label{prop:measurable_conjugacy}
    $\Phi$ is a Borel isomorphism. In addition, for any $a\in \mathcal{C}_v$ we have that $F^a\circ \Phi = \Phi \circ a$.
\end{prop}

\begin{proof}
    It is clear that $\Phi$ is a Borel isomorphism. We calculate
    \begin{align*}
        \Phi \circ a (x,y) & = \Phi(ax\gamma_x(a)^{-1}, \gamma_x(a)y) \\ 
        & = (ax\gamma_x(a)^{-1}, \gamma_{ax\gamma_x(a)^{-1}}(u_{ax})\gamma_x(a)y) \\
        & = (ax\gamma_x(a)^{-1}, \gamma_x(u_{ax}a)y) \\
        \text{and}\\
        F^a \circ \Phi(x,y) & = F^a(x,\gamma_x(u_x)y) \\
        & = (ax\gamma_x(a)^{-1}, F_{x}^a(\gamma_x(u_x)y)) \\
        & = (ax\gamma_x(a)^{-1}, \gamma_{\psi(x)}(u_{a\psi(x)}a)\gamma_x(u_x)y) \\
        & = (ax\gamma_x(a)^{-1}, \gamma_{x}(u_{a\psi(x)}au_x)y).
    \end{align*}
    So it suffices to show that 
    \begin{equation*}
        u_{a\psi(x)}au_x = u_{ax}a.
    \end{equation*}
    Recall that by definition, $u_{ax}ax = \psi(ax)$, whereas 
    \begin{equation*}
        u_{a\psi(x)}au_xx = u_{a\psi(x)}a\psi(x) = \psi(a\psi(x))
    \end{equation*}
    However since $\xi$ is $a$-decreasing, we have that $a\psi(x) \in a[x]_\xi \subset [ax]_\xi$ and hence $\psi(ax) = \psi(a\psi(x))$, completing the proof.
\end{proof}

Now let $\rho:x\mapsto \mu_x'$ be the disintegration of $\mu'$ with respect to the partition of $G/\Gamma \times M$ into fibers over $G/\Gamma$; so for a.e. $x\in G/\Gamma$, $\mu'_x$ is a measure on $M$. We have the following properties, analogous to \cite[Proposition 6]{brown_c1_2022}.

\begin{prop}
\label{prop:measurability_properties}
    We have for $a\in \mathcal{C}_v$ that
    \begin{enumerate}[label = {(\alph*)}]
        \item $x \mapsto F_x^a$ is $\xi$-measurable. In particular for a.e. $x\in G/\Gamma$ and for every $x'\in [x]_\xi$, $F_x^a = F_{x'}^a$;
        \item The function $x \mapsto \log \|(F_x^a)^{-1}\|_{C^1}$ is in $L^1(G/\Gamma,m_{G/\Gamma})$; 
        \item The fiberwise Lyapunov exponents for $a$ with respect to $\mu$ are all non-positive if and only if the fiberwise Lyapunov exponents for $F^a$ with respect to $\mu'$ are all non-positive;
        \item $\mu$ is $U$-invariant if and only if $\rho$ is $\xi$-measurable.
    \end{enumerate}
\end{prop}

\begin{proof}
    (a) is true by construction, (b) by cocompactness of $\Gamma$, and (c) by Proposition~\ref{prop:measurable_conjugacy}. For (d), first note that if $\mu$ is $U$-invariant then $\rho$ is clearly $\xi$-measurable. 
    
    Conversely, suppose $\rho$ is $\xi$-measurable. Then for $\mu$-a.e. $(x,y) \in G/\Gamma\times M$, for any $u \in U$ such that $ux \in [x]_\xi$, we have that $\mu_x = \mu_{ux}$. By standard properties of conditional measures, we have that $(u_*\mu)_x = \mu_{ux}$ (see \cite[Corollary 5.24]{EinsiedlerWard2011}). 
    
    Since $\mu$ is $a$-invariant we also have that $x \mapsto ((a^{-n})_*\mu)_x$ is $\xi$-measurable, and so $x\mapsto \mu_x$ is $a^{-n}\xi$-measurable. Since $ \bigcup_{n\geq 0}a^{-n}\xi = U$ we repeat this argument for all $n \geq 0$ to obtain that $\mu$ is $U$-invariant.
\end{proof}

\subsection{Avila-Viana Invariance Principle} We recall first the general setup. As before let $(X,\mathcal{B}_X,\nu_X)$ be a Lebesgue space, and $f:X\to X$ an invertible $\nu_X$-preserving measurable transformation. Let $M$ be a compact Riemannian manifold, and consider the trivial bundle $\mathcal{E} = X \times M$ with the projection map $\pi:\mathcal{E}\to X$.

\myindent Let $F:\mathcal{E}\to \mathcal{E}$ be a $\mathcal{B}_X\otimes \mathcal{B}_M$-measurable transformation. We say furthermore that it is a \emph{smooth cocycle over $f$} if it is of the form $F(x,y) = (f(x),F_x(y))$, where $F_x$ is a $C^1$-diffeomorphism of $M$ for each $x$. Denote by $F^{k}_x$ the $M$-component of $F^k$ over $x\in X$. More precisely, we set 
\begin{equation*}
    F_x^{k} := \begin{cases}
        F_{f^{k-1}(x)}\circ\cdots\circ F_x & \text{for } k\geq 0 \\
        (F_{f^{-k}(x)})^{-1}\circ \cdots\circ (F_{f^{-1}(x)})^{-1} & \text{for } k < 0.
    \end{cases}
\end{equation*}
This should not be confused with $(F_x)^k$.

Suppose that $F$ satisfies a mild integrability assumption, that is
\begin{equation}
\label{eq:L_1_of_inverse}
    \int |\log (\sup_{y\in M} \|D_yF_x^{-1}\|_{\operatorname{op}})|\, d\nu_X(x) < \infty.
\end{equation}

In this case, for any $F$-invariant probability measure $\mu$ on $\mathcal{E}$ such that $\pi_*\mu = \nu_X$, we can define the \emph{minimal Lyapunov exponent} via
\begin{equation*}
    \lambda_-((x,y),\mu,F) = \lim_{n\to \infty} \frac{1}{n}\log\|(D_yF_x^{n})^{-1}\|_{\operatorname{op}}^{-1}.
\end{equation*} 

This is well-defined $\mu$ almost everywhere, and is completely analogous to \eqref{eq:average_top_lyap}. The following result is \cite[Theorem 2]{brown_c1_2022}, which is a version of \cite[Theorem B]{avila_extremal_2010} only requiring the condition \eqref{eq:L_1_of_inverse}.

\begin{theo}[Avila-Viana invariance principle]
\label{theo:Avila-Viana}
    Let $F$ be a smooth cocycle over $f$ satisfying \eqref{eq:L_1_of_inverse}, and let $\mu$ be an $F$-invariant measure on $\mathcal{E}$ which projects under $\pi$ to $\nu_X$. Suppose that $\mathcal{B}_\xi$ is a $\sigma$-algebra on $X$ that generates $\mathcal{B}_X$ (mod 0) under $f$. 
    Suppose further that $f$ and $x \mapsto F_x$ are $\mathcal{B}_\xi$-measurable, and that $\lambda_-((x,y),\mu,F) \geq 0$ for $\mu$-a.e. $(x,y)$. Then the disintegration $x \mapsto \mu_x$ is $\mathcal{B}_\xi$ measurable.
\end{theo}

\subsection{Proof of Theorem~\ref{theo:invariance_under_stanble_and_unstable_horosphericals}}

In order to be in a position to apply Theorem~\ref{theo:Avila-Viana}, we need to be able to find zeroes of the Lyapunov exponents $\lambda_i(\cdot,\mu,\mathcal{D})$, but there is no reason for there to be any in $A \cong \Z^r$. For this purpose we extend our action to one of $\R^r$.

Consider $\R^r \times M^\alpha \cong \R^r\times G/\Gamma \times M$. This admits a left $\R^r$-action that is trivial in the $M^\alpha$ component, and a right $\Z^r$-action via
\begin{equation*}
    (v,x,y) \cdot a := (v+a,a^{-1}\cdot(x,y))
\end{equation*}
As before, these two actions commute and so descend to an action of $\R^r$ on 
\begin{equation*}
    \widehat{M^\alpha}:= (\R^r\times M^\alpha)/\Z^r.
\end{equation*}
This is a bundle over $\R^r/\Z^r$ with fibers $M^\alpha$. After identifying $\widehat{M^\alpha}$ measurably with $\R^r/\Z^r \times M^\alpha$, we equip it with the measure $m_{\R^r/\Z^r}\otimes \mu$, where $m_{\R^r/\Z^r}$ denotes the Haar measure and $\mu$ is our $A$-invariant and ergodic measure that projects to the Haar measure $m_{G/\Gamma}$.

\myindent We extend the cocycle $\mathcal{D}$ to a cocycle $\widehat{\mathcal{D}}$ for the $\R^r$-action over $\widehat{M^\alpha}$, via the formula
\begin{equation*}
    \widehat{\mathcal{D}}_{(v,x,y)}w := \mathcal{D}_{(x,y)}(\lfloor v+w\rfloor) \quad \text{for } (v,x,y) \in \widehat{M^\alpha}, w\in \R^r.
\end{equation*}
Here $\lfloor\cdot\rfloor:\R^r\to \Z^r$ is the floor function, which serves the role of the return cocycle with respect to the fundamental domain $[0,1)^r$. It is clear by definition that for $a\in \Z^r \leq \R^r$ and $m_{\R^r/\Z^r}\otimes \mu$-a.e. $(v,x,y) \in \widehat{M^\alpha}$ we have that
\begin{equation}
    \label{eq:cocycles_agree}
    \widehat{\mathcal{D}}_{(v,x,y)}a  = \mathcal{D}_{(x,y)}a
\end{equation}

We can consider the Lyapunov cocycles of $\widehat{\mathcal{D}}$ given by Theorem~\ref{theo:Oseledets}. Since the $\R^r$-action on $\R^r/\Z^r$ is ergodic they are a.e. constant, and by \eqref{eq:cocycles_agree} we in fact have (up to relabelling) that 
\begin{equation*}
    \lambda_i(\cdot,\mu,\mathcal{D}) = \lambda_i(\cdot,m_{\R^r/\Z^r}\otimes \mu,\widehat{\mathcal{D}})|_{\Z^r}.
\end{equation*}

Consider the $\sigma$-algebra $\mathcal{B}_{\R^r/\Z^r}\otimes \mathcal{B}_\xi$ on $\R^r/\Z^r\times G/\Gamma$, and the Borel Isomorphism $\widehat{\Phi}:\widehat{M^\alpha} \to \R^r/\Z^r \times G/\Gamma \times M$ where $\widehat{\Phi} := \Id \times \Phi$.

\begin{proof}[Proof of Theorem~\ref{theo:invariance_under_stanble_and_unstable_horosphericals}]
    If $\dim M < \dim A$, then there is some 
    \begin{equation*}
        v\in \bigcap_i \ker \lambda_i(\cdot,m_{\R^r/\Z^r}\otimes \mu,\widehat{\mathcal{D}}).
    \end{equation*} 
    If $\dim M = \dim A$ and the action is volume preserving then the Lyapunov exponents sum to zero, so again we can find some $v$ in the common kernels. Either way, there is some $v$ with $\lambda_-(v,m_{\R^r/\Z^r}\otimes \mu,\widehat{\mathcal{D}}) = 0$. We are now in a position to apply Theorem~\ref{theo:Avila-Viana}, and by Proposition~\ref{prop:measurability_properties} we conclude that $\mu$ is $G_v^-$-invariant. By repeating the same argument for $-v$ we conclude that $\mu$ is $G_v^+$-invariant, giving the Theorem.
\end{proof}
	\printbibliography[title=References]
\end{document}